\documentclass[10pt]{article}
\usepackage{appendix}
\usepackage{mathrsfs}
\usepackage{amsmath}
\usepackage{amsfonts}
\usepackage{enumerate}
\usepackage{setspace}
\usepackage{bbm}
\usepackage{amsthm}
\usepackage{amssymb}
\usepackage{dsfont}
\usepackage{bm}
\usepackage{BOONDOX-cal}
\usepackage{empheq}
\usepackage{authblk}
\numberwithin{equation}{section}

\newtheorem{theorem}{Theorem}[section]
\newtheorem{example}[theorem]{Example}
\newtheorem{definition}[theorem]{Definition}
\newtheorem{lemma}[theorem]{Lemma}
\newtheorem{remark}[theorem]{Remark}
\newtheorem{assumption}[theorem]{Assumption}

\newcommand{\abs}[1]{\left\vert#1\right\vert}

\makeatletter
\def\blfootnote{\gdef\@thefnmark{}\@footnotetext}
\makeatother

\title{Connection between MP and DPP for Mean-Field Stochastic
	Optimal Control Problems}
\author{
	Guangzhi Cui\textsuperscript{1} \qquad 
	Tianyang Nie\textsuperscript{1,*}\qquad 
	Guangchen Wang\textsuperscript{2}}

\date{}

\begin{document}
\maketitle
\blfootnote{\scriptsize \textbf{Funding:} This work was supported by National Key Research and Development Program of China [Grant 2022YFA1006104], the
	National Natural Science Foundation of China [Grants 12371447 and 61925306], and the
	Shandong Provincial Key Laboratory of Stochastic System Control and Scientific Computing.}

\blfootnote{\scriptsize \rlap{\textsuperscript{1}} \hspace{0.3em} School of Mathematics, Shandong University, Jinan 250100, PR China.}
\blfootnote{\scriptsize \rlap{\textsuperscript{2}} \hspace{0.3em} School of Control Science and Engineering, Shandong University, Jinan 250061, PR China.}

\blfootnote{\scriptsize \rlap{*} \hspace{0.3em} Corresponding author.}

\blfootnote{\scriptsize \textbf{Email:} 202111966@mail.sdu.edu.cn (G. Cui), nietianyang@sdu.edu.cn (T. Nie), wguangchen@sdu.edu.cn (G. Wang).}

\begin{abstract}
	This paper deals with the optimal control problem for the general mean-field stochastic differential equations, in which the coefficients depend on both the state process as well as on its law. We focus on the connection between the general maximum principle and the dynamic programming principle for such a control problem. We establish the relationship between the value function of the mean-field optimal control problem and the adjoint equations in the maximum principle. We also give a general mean-field stochastic verification theorem. Some examples are solved explicitly to illustrate our results.
\end{abstract}
\noindent
\textbf{MSC:}  49K45, 49L20,  60H10, 93E20.
\vspace{5pt}\\
\textbf{Keywords:} Mean-field optimal control, Maximum principle, Dynamic programming principle, Verification theorem.

\section{Introduction}
There are generally two methods to study the optimal control problem, one is the Pontryagin’s maximum principle (MP), the other one is the dynamic programming principle (DPP). The connection between MP and DPP of the deterministic optimal control problem can be found in Fleming and Rishel \cite{fleming2012deterministic}. They considered the classical case that the value function is smooth, and obtained that $\psi(t)=-V_x(t, \bar{x}(t))$ and $ V_t(t, \bar{x}(t))=H(t, \bar{x}(t), \bar{u}(t), \psi(t))$, where $\psi$, $H$, $V$, $\bar{u}(t)$ and $\bar{x}(t)$ represent adjoint variable, Hamiltonian function, value function, optimal control and optimal state respectively. In the framework of the viscosity solutions, Barron and Jensen \cite{barron1986pontryagin} first used the viscosity solution to derive the MP from DPP. Zhou \cite{zhou1990maximum} further employed  the notions of subdifferential and supdifferential to reveal their relationship.
For stochastic optimal control problems,  Bensoussan \cite{bensoussan2006lectures}, Yong and Zhou \cite{yong2012stochastic} described the connection between DPP and MP for the classical case, and  Zhou \cite{zhou1990connection,zhou1991unified} further elaborated the relevant results under the framework of viscosity solutions. For stochastic recursive optimal control problem, where the cost functional is described by the solution to a  backward stochastic differential equation (BSDE), we revealed the connection between DPP and MP both under classical framework and viscosity solution framework in our paper \cite{nie2017connection}.\par
In this paper, we consider a class of stochastic control problems for the state processes governed by general McKean-Vlasov equations (also known as mean-field stochastic differential equations) in which the coefficients depend nonlinearly on both the state process and its law in the following form:
\begin{equation}
	X_{t}=\xi+\int_{s}^{t} b\left(r,X_{r},\mathbb{P}_{X_{r}},u_{r}\right)dr+\int_{s}^{t} \sigma\left(r,X_{r},\mathbb{P}_{X_{r}},u_{r}\right)dB_r, \quad t\in \left[s,T\right].\label{intro:1}
\end{equation} 
The pioneering works on the McKean-Vlasov equations are attributed to McKean \cite{mckean1967propagation} and Kac \cite{kac1956foundations}, where they studied uncontrolled SDEs and established the results of propagation of chaos. There has been risen dramatically over the past decade in this topic, due to the connection that it shares with the so-called mean-field game (MFG) theory, introduced independently by Lasry and Lions in \cite{lasry2007mean}, and Huang, Caines and Malhamé in \cite{huang2007nash}. The McKean-Vlasov equations arise naturally when one tries to describe the behavior of many symmetric agents that interact through an empirical distribution of their states. Furthermore, mean-field games arise when the concept of Nash's non-cooperative equilibrium is employed, while mean-field control problems are related to Pareto optimality where players can be identified with a single ``representative agent''. Indeed, these two problems have subtle differences, which were thoroughly investigated by Carmona, Delarue and Lachapelle \cite{carmona2013control}. It's worth noting that Bensoussan, Frehse and Yam \cite{bensoussan2013mean}, Carmona and Delarue \cite{carmona2018probabilistic} provide an overview of recent developments on these two issues.\par 
Our goal is to study the optimal control of McKean-Vlasov equations and to establish the connection between MP and DPP. As is well known, there is much literature about the optimal control problem of McKean-Vlasov equations solved through DPP or MP in recent years. For example, Andersson and Djehiche \cite{andersson2011maximum}, and Buckdahn, Djehiche and Li \cite{buckdahn2011general} studied the MP for the mean-field optimal control problem where the state and the cost functional depend on the expectation of the state process, respectively, with convex control domain and with non-convex control domain. Then, Buckdahn, Li and Ma \cite{buckdahn2016stochastic} extended the results to the general case with non-convex control domain, where the coefficients depend on the law of the state process rather than its expectation, and obtained the corresponding MP. Recently, we studied mean-field partial information non-zero sum stochastic
differential games and obtained the MP by modifying the related second-order adjoint equation in our paper \cite{Nie-Yan-Mean-field-partial-information-non-zero-sum}. For DPP method, due to the nonlinear dependency with respect to the law of the process, the problem is actually time-inconsistent, and Bellman’s optimality principle does not hold. However, although the problem itself is time-inconsistent, some form of DPP can be restored by  changing the state space to density function space or probability measure space. Lauri\'ere and Pironneau \cite{lauriere2014dynamic} first implemented the idea by assuming the existence at all times of a density for the marginal distribution of the state process and reformulated the problem as a density control problem with a family of deterministic controls. Later, Bensoussan, Frehse and Yam \cite{bensoussan2013mean,bensoussan2015master,bensoussan2017interpretation} developed this method further to investigate the well-posedness of the master equation. Under this reformulation they managed to prove a DPP and deduce a Hamilton-Jacobi-Bellman (HJB) equation in the space of density functions. Without the density assumption, most of the current research took the distribution of state process as the variable of value function to achieve the DPP, see e.g., Pham and Wei \cite{pham2018bellman}, Cosso and Pham \cite{cosso2019zero}. In the case with common noise, Pham and Wei \cite{pham2017dynamic} proved a DPP for the case where the control process is adapted to the common noise filtration. Then, Djete, Possama\"i and Tan \cite{djete2022mckean} investigated the path-dependent mean-field optimal control problem and established the DPP under both strong and weak formulations. Later, Cosso et al. \cite{cosso2023optimal} proved the DPP in the infinite-dimensional, non-Markov case.\par
Due to the different technologies used in the MP and DPP for the mean-field optimal control problem, their research and development are relatively independent. There is little research on the connection between the two approaches. To our knowledge, only Carmona and Delarue used some connection in the heuristic sense between DPP and MP of the mean-field optimal control problem in their monograph \cite{carmona2018probabilistic}, where they assumed that the diffusion term $\sigma$ does not depend on the control. Through a combination of probabilistic and analytic approaches, they transformed the mean-field control problem into a standard stochastic control problem by introducing nonlinear Kolmogorov-Fokker-Planck equation and the functional derivative. Then they got the result (see  \cite{carmona2018probabilistic} Volume I, Section 6.2, equation (6.45)):
\begin{equation}
	\left(p_t,q_t\right)_{0\le t\le T}=\left(\partial_{x}u\left(t,X_t\right),\partial_{xx}^2u\left(t,X_t\right)\sigma\left(t,X_t,\mu_t\right)\right)_{0\le t\le T}\label{intro:2},
\end{equation} 
where $\left(p_t,q_t\right)$ solves the related adjoint equation, and $u\left(t,x\right)=\dfrac{\delta v}{\delta m}\left(t,\mu_t\right)\left(x\right)$. Here, $\dfrac{\delta v}{\delta m}$ denotes the linear functional derivative of the value function $v\left(t,\mu_t\right)$ w.r.t the measure, and $\mu_t$ is the marginal
distribution of the optimal path $X_t$ which is given by the nonlinear Kolmogorov-Fokker-Planck equation. Later, by using the pure probabilistic method (see Volume I, Section 6.5 of \cite{carmona2018probabilistic}), the authors defined a new value function $V(t,x,\mu)$ satisfying $v(t,\mu)=\int_{\mathbb{R}^{d}}V(t,x,\mu)\mu(dx)$. By comparing the adjoint equations derived respectively from two value functions $v$ and $V$, they obtained
\begin{equation}
	\left(p_t,q_t\right)_{0\le t\le T}=\left(\partial_{\mu}v\left(t,\mathbb{P}_{X_t};X_t\right),\partial_{x}\partial_{\mu}v\left(t,\mathbb{P}_{X_t};X_t\right)\sigma\left(t,X_t,\mathbb{P}_{X_t}\right)\right)_{0\le t\le T},
\end{equation}
which can also be derived by (\ref{intro:2}) and the relationship between L-derivative and linear functional derivative (see page 534 of \cite{carmona2018probabilistic}). We mention that the connection obtained above is just at a heuristic level and in the case that $\sigma$ does not depend on the control as well as the controls have a feedback form. Thus the aim of this paper is to establish the relationship between MP and DPP for a more general system in which the diffusion term $\sigma$ can depend on the control, and moreover, the controls are not in feedback form. The main difficulties and contributions of our paper are as follows:\par
\begin{itemize}
	\item Firstly, for the study on the connection between MP and DPP, the control system considered in this paper is more general. Different from \cite{bensoussan2006lectures,yong2012stochastic,zhou1990connection}, our system coefficients can depend on the law of state process. Compared with \cite{carmona2018probabilistic}, the diffusion term $\sigma$ in our paper is allowed to depend on the control, and the controls are not in feedback form. Moreover, there is currently little literature for mean-field stochastic verification theorem, for example, \cite{pham2017dynamic} and \cite{pham2018bellman} give the verification theorem of the system in the case with feedback control. By using It\^o's formula along the flow of probability measures and HJB equation (\ref{hjb:1}), we establish a general mean-field stochastic verification theorem that the controls are not necessarily in feedback form.

	\item Secondly, we establish the connection between the first-order adjoint equation of MP and the value function. Compared with  \cite{bensoussan2006lectures,fleming2012deterministic,nie2017connection,yong2012stochastic,zhou1990connection},  the value function is defined on the space of probability measures, and the first-order adjoint equation is a mean-field type BSDE in this paper, thus the classical method is no longer applicable. By using the DPP and the verification theorem obtained in this paper, we derive that the law of optimal state maximizes the left hand side of (\ref{proof:1}) over the space of probability measures (see Theorem \ref{classical}, Theorem \ref{ver} and the proof of Theorem \ref{Th;4.1}), then through perturbing the law of optimal state, we get the first- and second-order maximum conditions (see (\ref{first-m-c}) and (\ref{Bianfen:2})). In order to establish the relationship between the first-order adjoint equation and the value function, we analyze the first-order maximum condition and use Itô's formula to obtain the BSDE for $\partial_{\mu}v(t,\mathbb{P}_{X_t^{\ast}};X_t^{\ast})$, we also need to to reformulate the BSDE for $\partial_{\mu}v(t,\mathbb{P}_{X_t^{\ast}};X_t^{\ast})$ into a form that is similar to the first-order adjoint equation. Unfortunately, we observe that the high-dimensional higher-order L-derivative, which introduces new variables into the independently copied probability space, renders the situation more complex. For example, when analyzing (\ref{first-m-c}) and applying Fubini's theorem to it, we observe that the term $\widetilde{\mathbb{E}}\left[\widetilde{v}_{\mu\mu}^{\ast}\left(t\right)^{\top}\widetilde{b}\left(t\right)\right]$ in (\ref{first zhankai}) is different from the term $\widetilde{\mathbb{E}}\left[\widetilde{v}_{\mu\mu}\left(t\right)\widetilde{b}\left(t\right)\right]$ in (\ref{Th1:dvmu}). More precisely, the positions of the two state variables $X_t^{\ast}$ and $\widetilde{X}_t^{\ast}$ (one from the original probability space and the other from the independently copied probability space) in $v_{\mu\mu}$ are reversed between these two terms, and the $v_{\mu\mu}$ in the former one is the transpose of that in the latter one. The similar situation can be observed in other terms, as well as in the proof of the second-order relationship. To overcome the aforementioned difficulties, we obtain the partial-derivative commutativity properties of L-derivative, and demonstrate the symmetry of the higher-order matrix-valued derivative as well as the exchangeability of variables in the second-order L-derivative between the original probability space and the independently copied probability space, see Lemma \ref{lem:1}, Lemma \ref{lem:2} and Lemma \ref{lem:3}. Finally, we use the uniqueness of solution for mean-field type BSDE to obtain the relationship between the value function and the first-order adjoint equation.

	\item Thirdly, we get the relationship between the second-order adjoint equation of MP and the value function. To our knowledge, this result has not been addressed in the existing literature. The difficulty we encounter here is that higher-order L-derivative naturally introduces numerous additional complex terms (see e.g. \eqref{bianfen:2.1}), and these terms obscure the underlying structure of the second-order maximum conditions, making the comparison theorem for BSDEs not directly applicable. To handle this difficulty, we construct independent random variables with the specifically designed distributions to eliminate these complex terms, see Remark \ref{rem:independ} and the proof of Theorem \ref{TH4.2}. By using a modified second-order adjoint equation introduced in \cite{Nie-Yan-Mean-field-partial-information-non-zero-sum} and comparison theorem for scalar BSDEs, we obtain the connection between the second-order partial derivative of the value function and the second-order adjoint equation. We mention that since it is hard to verify the conditions of the comparison theorem for matrix-valued BSDEs, our result may only work for 1-dimensional case, see Remark \ref{matrix-BSDE-difficult} for more details. We hope we can solve the multidimensional case in the future.
	\item Finally, we give some examples which can be solved explicitly, see Example \ref{exa:1} and Example \ref{exa:2}. Compared with \cite{nie2017connection,yong2012stochastic}, our examples are more complicated due to the appearance of the mean-field terms. Furthermore, we emphasize that, through two different perspectives, these examples provide the underlying reasons for when equality and strict inequality hold in the inequality \eqref{th:2} of Theorem 4.2. 
\end{itemize}

This paper is organized as follows. In Section \ref{section:2}, we recall the notion of L-derivative and formulate the mean-field optimal control problem. In Section \ref{section3}, we
give some preliminary results about the MP and DPP, and provide a more general verification theorem. In Section \ref{section4}, we derive the connection of MP and DPP, and give some examples to explain the results.

\section{Preliminaries and Problem statement}
\label{section:2}

\subsection{Preliminaries}\label{sec:2.1}
Now let us briefly introduce some notations and tools, especially the concept of the derivative of probability measures in mean-field theory and the It\^o's formula along a flow of probability measures.\par
In this paper, let $T>0$ denote an arbitrarily fixed time horizon. 
We consider a complete probability space $(\Omega, \mathcal{F}, \mathbb{P})$, on which we define a m-dimensional Brownian motion $B=\left\{B_t\right\}_{t \in[0, T]}$. We denote by $\mathbb{F}^B=\left(\mathcal{F}_t^B\right)_{t \geq 0}$ the $\mathbb{P}$-completion of the filtration generated by $B$, which is right-continuous, so that it satisfies the usual conditions. For a generic Euclidean space $\mathbb{X}$, we denote its inner product by `` $\cdot$ '', its norm by $|x|$ for any $x\in\mathbb{X}$, and its Borel $\sigma$-field by $\mathcal{B}(\mathbb{X})$. Then for any sub-$\sigma$-field $\mathcal{H} \subseteq \mathcal{F}$ and $1 \leq p<\infty$, we denote by $L^p(\mathcal{H} ; \mathbb{\mathbb{R}}
^d)$ the space of all $\mathbb{\mathbb{R}}^d$-valued, $\mathcal{H}$-measurable random variables $\xi$ with $\|\xi\|_p :=$ $\mathbb{E}\left[|\xi|^p\right]^{1 / p}<\infty$, and denote by $L^{\infty}(\mathcal{H} ; \mathbb{R}
^d)$ the space of all $\mathbb{\mathbb{R}}^d$-valued, $\mathcal{H}$-measurable random variables $\xi$ with $\|\xi\|_{\infty} :=$ $\operatorname{ess\,sup}_{\omega\in\Omega}\abs{\xi(\omega)}<\infty$. In particular, $L^2\left(\mathcal{H} ; \mathbb{R}^d\right)$ is the Hilbert space with inner product $\left\langle\xi, \eta\right\rangle \triangleq \mathbb{E}[\xi \cdot \eta]$, for any $\xi, \eta \in L^2\left(\mathcal{H} ; \mathbb{R}^d\right)$. $\mathcal{P}_2\left(\mathbb{R}^d\right)$ denotes the space of all probability measures $\mu$ over $\left(\mathbb{R}^d, \mathcal{B}\left(\mathbb{R}^d\right)\right)$ with finite second moment, i.e. $\int_{\mathbb{R}^d}|x|^2 \mu(d x)<\infty$. It is noteworthy that the space $\mathcal{P}_2\left(\mathbb{R}^{d}\right)$ is endowed with the 2-Wasserstein metric:
$$
\begin{aligned}
	W_2(\mu, v):= & \inf \left\{\left(\int_{\mathbb{R}^d \times \mathbb{R}^d}|x-y|^2 \rho(d x d y)\right)^{1 / 2}, \rho \in \mathcal{P}_2\left(\mathbb{R}^{2 d}\right)\right. \\
	& \qquad\text { with } \rho\left(\cdot \times \mathbb{R}^d\right)=\mu,\text{ }\rho\left(\mathbb{R}^d \times \cdot\right)=\nu\Bigg\}, \quad \mu, \nu\in \mathcal{P}_2\left(\mathbb{R}^d\right) .
\end{aligned}
$$
To ensure that there is a corresponding random variable for each probability measure in space $\mathcal{P}_2\left(\mathbb{R}^{d}\right)$, we assume that there exists a sub-$\sigma$-algebra $\mathcal{G}$ of $\mathcal{F}$ satisfying the following properties:
\begin{enumerate}[(i)]
	\item $\mathcal{G}$ and $\mathcal{F}_{\infty}^B$ are independent.
	\item  $\mathcal{G}$ is ``rich enough", namely $\mathcal{P}_2\left(\mathbb{R}^d\right)=\left\{\mathbb{P}_{\xi} ,\xi \in L^2\left(\mathcal{G};\mathbb{R}^d \right) \right\}$.
\end{enumerate}
We denote by $\mathbb{F}=\left(\mathcal{F}_t\right)_{t \geq 0}$ the filtration defined as
$
\mathcal{F}_t=\mathcal{G} \vee \mathcal{F}_t^B, t \geq 0.
$ We observe that $\mathbb{F}$ satisfies the usual conditions of $\mathbb{P}$-completeness and right-continuity. 
\begin{remark}\label{rem:independ}
The ``rich enough'' condition ensures the existence on $\mathcal{G}$ of some independent random variables 
$\xi_1$, $\xi_2$, $\cdots$, with arbitrarily prescribed distributions. More precisely, for any given $\mu_1,\mu_2,\cdots$ on some Borel spaces $S_1,S_2,\cdots$, there exists some independent random variables $\xi_1,\xi_2,\cdots$ on $(\Omega,\mathcal{G})$ with distributions $\mu_1,\mu_2,\cdots$.\par
Indeed, assume that $\mathcal{G}$ is ``rich enough" and denote by $\lambda$ the Lebesgue measure on $([0,1],\mathcal{B}([0,1]))$. We can find a random variable $\Gamma_{G}:(\Omega,\mathcal{G})\rightarrow([0,1],\mathcal{B}([0,1]))$ with uniform distribution $\lambda$. In addition, following from Theorem 4.19 of \cite{kallenberg2021foundations}, for the given $\mu_1,\mu_2,\cdots$ on some Borel spaces $S_1,S_2,\cdots$, there exists some mutually independent random variables $\eta_1,\eta_2\cdots$ on $([0,1],\mathcal{B}([0,1]))$ such that $\eta_i:[0,1]\rightarrow S_i$ has distribution $\mu_i$, for $i\ge1$. Then let $\xi_i=\eta_i(\Gamma_G):\Omega\rightarrow S_i$, we state that $\xi_1$, $\xi_2$, $\cdots$ are mutually independent random variables with distribution $\mu_1,\mu_2,\cdots$. In fact, we can get that for any $A_i\in \mathcal{B}(S_i)$, $i\ge1$,
\begin{equation*}
	\mathbb{P}(\xi_i\in A_i)=\mathbb{P}(\eta_i(\Gamma_G)\in A_i)=\mathbb{P}(\Gamma_G\in \eta_i^{-1}(A_i))=\lambda(\eta_i^{-1}(A_i))=\lambda(\eta_i\in A_i)=\mu_i(A_i).
\end{equation*}
Furthermore, for any finite subset $\left\{i_1,\cdots,i_n\right\}\subset\mathbb{N}$ and $B_{i_k}\in \mathcal{B}(S_{i_k})$, $1\le k\le n$, we have 
\begin{equation*}
	\begin{aligned}
		&\mathbb{P}(\xi_{i_1}\in B_{i_1},\cdots,\xi_{i_n}\in B_{i_n})=\mathbb{P}(\eta_{i_1}(\Gamma_G)\in B_{i_1},\cdots,\eta_{i_n}(\Gamma_G)\in B_{i_n})\\
		=&\mathbb{P}(\Gamma_G\in\bigcap_{1\le k\le n}\eta_{i_k}^{-1}(B_{i_k}))=\lambda(\bigcap_{1\le k\le n}\eta_{i_k}^{-1}(B_{i_k}))=\prod_{\substack{1\le k\le n}}\lambda(\eta_{i_k}^{-1}(B_{i_k}))=\prod_{\substack{1\le k\le n}}\mathbb{P}(\xi_{i_k}\in B_{i_k}),
	\end{aligned}
\end{equation*}
 which means that $\xi_1$, $\xi_2$, $\cdots$ are mutually independent random variables.
\end{remark}
\par
Next, we will introduce the differentiability of the probability measure. This approach was pioneered by Lions \cite{lionscours} (the readers can also refer to \cite{carmona2018probabilistic}).

Under the assumption that the probability space $(\Omega, \mathcal{F}, \mathbb{P})$ is rich enough. For any function $f: \mathcal{P}_2\left(\mathbb{R}^d\right) \mapsto \mathbb{R}$, we consider its lift function $f^{\sharp}: L^2\left(\mathcal{F} ; \mathbb{R}^d\right) \mapsto \mathbb{R}$, such that $f^{\sharp}(\theta):=f\left(\mathbb{P}_{\theta}\right)$, which depends only on the law of $\theta \in L^2\left(\mathcal{F} ; \mathbb{R}^d\right)$. We say that a function $f: \mathcal{P}_2\left(\mathbb{R}^d\right) \mapsto \mathbb{R}$ is differentiable at $\mu_0 \in \mathcal{P}_2\left(\mathbb{R}^d\right)$ if its lift $f^{\sharp}$ is Fréchet differentiable at $\theta_0$ with $\theta_0 \in L^2\left(\mathcal{F}, \mathbb{R}^d\right)$ and $\mathbb{P}_{\theta_0}=\mu_0$. More exactly, there exists a continuous linear functional $D f^{\sharp}\left(\theta_0\right): L^2\left(\mathcal{F} ; \mathbb{R}^d\right) \mapsto \mathbb{R}$ such that
$
	f^{\sharp}\left(\theta_0+\eta\right)-f^{\sharp}\left(\theta_0\right)= \left[D f^{\sharp}\left(\theta_0\right)\right]\left(\eta\right) +o\left(\|\eta\|_2\right) \triangleq D_\eta f\left(\mu_0\right)+o\left(\|\eta\|_2\right), 
$
here we denote  the Fréchet derivative of $f$ at $\mu_0$, in the direction $\eta$, by $D_\eta f\left(\mu_0\right)$.\par
In this case, $D_\eta f\left(\mu_0\right)$ which viewed as an element of $L^2\left(\mathcal{F} ; \mathbb{R}^d\right)$, by the Riesz' representation theorem, can be represented as
$
\left[D f^{\sharp}\left(\theta_0\right)\right]\left(\eta\right)=\left\langle\Theta_0, \eta\right\rangle=\mathbb{E}\left[\left(\Theta_0, \eta\right)\right] $, for $\eta \in L^2(\mathcal{F} ; \mathbb{R}^d),
$
where $\Theta_0 \in L^2\left(\mathcal{F} ; \mathbb{R}^d\right)$ is a ($\mathbb{P}\mbox{-}a.s.$) unique random variable. From \cite{cardaliaguet2010notes} or \cite{lionscours}, there exists a Borel function $h\left[\mu_0\right]: \mathbb{R}^d \mapsto \mathbb{R}^d$, which depends only on the law $\mu_0$ $(=\mathbb{P}_{\theta_0})$ but not on the particular choice of the representative $\theta_0$, such that $\Theta_0=h\left[\mu_0\right]\left(\theta_0\right)$. Thus we can denote $\partial_\mu f\left(\mathbb{P}_{\theta_0}; y\right) \triangleq h\left[\mu_0\right](y)$, $y \in \mathbb{R}^d$, and we have the following identities:
$
\Theta_0=h\left[\mathbb{P}_{\theta_0}\right]\left(\theta_0\right)=\partial_\mu f\left(\mathbb{P}_{\theta_0}; \theta_0\right),
$ and $D_\eta f\left(\mathbb{P}_{\theta_0}\right)=\left\langle\partial_\mu f\left(\mathbb{P}_{\theta_0}; \theta_0\right) ,\eta\right\rangle$. We note that for each $\mu \in \mathcal{P}_2\left(\mathbb{R}^d\right), \partial_\mu f\left(\mathbb{P}_{\theta}; \cdot\right)=h\left[\mathbb{P}_\theta\right](\cdot)$ is only defined in  $\mathbb{P}_{\theta}(d y)$-a.e. sense, where $\mu=\mathbb{P}_{\theta}$. In the following we give some useful definitions. 
\begin{definition}	
	We say that $f \in \mathbb{C}^{1}\left(\mathcal{P}_2\left(\mathbb{R}^d\right)\right)$ if for all $\theta \in L^2\left(\mathcal{F} ; \mathbb{R}^d\right)$, there exists a $\mathbb{P}_{\theta}$-modification of $\partial_\mu f\left(\mathbb{P}_{\theta}; \cdot\right)$, denoted by itself, such that $\partial_\mu f: \mathcal{P}_2\left(\mathbb{R}^d\right) \times \mathbb{R}^d \mapsto \mathbb{R}^d$ is continuous.
	\label{def:1}
\end{definition}
Let $f \in \mathbb{C}^{1}\left(\mathcal{P}_2\left(\mathbb{R}^d\right)\right)$, and we denote the mapping $\partial_\mu f:=([\partial_\mu f]_1, \cdots[\partial_\mu f]_d)^{\top}: \mathcal{P}_2\left(\mathbb{R}^d\right) \times \mathbb{R}^d \mapsto \mathbb{R}^d$. Throughout the paper, we also denote $\partial_{\mu}f:=(\partial_{\mu^1} f,\cdots, \partial_{\mu^d} f)^{\top}$ for brevity.
\begin{definition} We define $ \mathbb{C}^{2}\left(\mathcal{P}_2\left(\mathbb{R}^d\right)\right)$ as the set of continuous functions $f\in\mathbb{C}^{1}\left(\mathcal{P}_2\left(\mathbb{R}^d\right)\right)$, such that
	\begin{enumerate}[(i)]
		\item  $\left[\partial_\mu f\right]_i(\cdot ; x) \in \mathbb{C}^{1}\left(\mathcal{P}_2\left(\mathbb{R}^d\right)\right)$, for all $x \in \mathbb{R}^d,$ $1 \leq i \leq d$;
		\item $\partial_\mu f(\mu; \cdot): \mathbb{R}^d \mapsto \mathbb{R}^d$ is differentiable, for every $\mu \in \mathcal{P}_2\left(\mathbb{R}^d\right)$;
		\item $\partial_x \partial_\mu f: \mathcal{P}_2\left(\mathbb{R}^d\right) \times \mathbb{R}^d \mapsto \mathbb{R}^d \times \mathbb{R}^d$ and $\partial_\mu^2 f\left(\mu;x, y\right):=\partial_\mu\left[\partial_\mu f(\cdot; x)\right]\left(\mu; y\right)$ : $\mathcal{P}_2\left(\mathbb{R}^d\right) \times \mathbb{R}^d \times \mathbb{R}^d \mapsto \mathbb{R}^d \times \mathbb{R}^d$ are continuous.
	\end{enumerate}
\end{definition}
We now present the definition for the function $f:\mathbb{R}^d\times\mathcal{P}_2(\mathbb{R}^d)\mapsto\mathbb{R}$. The first-order (resp. second-order) derivative with respect to measure $\mu$ introduces a new variable, and we denote it by $y$ (resp. $z$) such that $\partial_{\mu}^2 f(x, \mu;y, z):\mathbb{R}^d\times\mathcal{P}_2(\mathbb{R}^d)\times\mathbb{R}^d\times\mathbb{R}^d \mapsto\mathbb{R}^d\times\mathbb{R}^d.$
\begin{definition}
	We denote by $f \in \mathbb{C}^{1,1}(\mathbb{R}^d \times \mathcal{P}_2(\mathbb{R}^d) )$ the set of continuous functions $f:\mathbb{R}^d\times\mathcal{P}_2(\mathbb{R}^d)\mapsto \mathbb{R}$, such that
		\begin{enumerate}[(i)]
		\item $ f(x, \cdot)\in \mathbb{C}^{1}\left(\mathcal{P}_2\left(\mathbb{R}^d\right)\right)$, for all $x \in \mathbb{R}^d$;
		\item  $ f(\cdot, \mu)$ is differentiable, for all $\mu \in \mathcal{P}_2(\mathbb{R}^d)$;
		\item all the first-order derivatives are continuous.
	\end{enumerate}
\end{definition}
\begin{definition}
	We denote by $f \in \mathbb{C}^{2,2}(\mathbb{R}^d \times \mathscr{P}_2(\mathbb{R}^d) )$ the set of continuous functions $f:\mathbb{R}^d\times\mathcal{P}_2(\mathbb{R}^d)\mapsto \mathbb{R}$, such that for $1\le i\le d$,
	\begin{enumerate}[(i)]
		\item $\partial_{x^i} f(\cdot, \cdot) \in \mathbb{C}^{1,1}\left(\mathbb{R}^d \times \mathcal{P}_2(\mathbb{R}^d)\right)$;
		\item $[\partial_{\mu} f]_{i}(\cdot, \cdot;y) \in \mathbb{C}^{1,1}\left(\mathbb{R}^d \times \mathscr{P}_2(\mathbb{R}^d)\right)$, for fixed $y\in\mathbb{R}^d$, and $[\partial_{\mu}f]_{i} (x, \mu;\cdot) \in \mathbb{C}^{1}\left( \mathbb{R}^d\right)$, for fixed $(x,\mu)\in \mathbb{R}^d\times\mathcal{P}_2(\mathbb{R}^d);$
		\item all the second-order derivatives are continuous.
	\end{enumerate}
\end{definition}
\begin{definition}
	We denote by $f \in \mathbb{C}_{b,loc}^{2,2}(\mathbb{R}^d \times \mathscr{P}_2(\mathbb{R}^d) )$ the set of continuous functions $f \in \mathbb{C}^{2,2}(\mathbb{R}^d \times \mathscr{P}_2(\mathbb{R}^d) )$ such that all derivatives of $ f$ are locally bounded in $(x,\mu)$ uniformly with respect to $(y,z)$.
\end{definition}
Next, we give the definitions for the function $f:[0,T]\times\mathcal{P}_2(\mathbb{R}^d)\mapsto\mathbb{R}$.
\begin{definition}
	We denote by $C^{1,2}([0,T]\times\mathcal{P}_2(\mathbb{R}^d))$ the set of continuous functions $f:[0,T]\times\mathcal{P}_2(\mathbb{R}^d)\mapsto \mathbb{R}$, such that 
	\begin{enumerate}[(i)]
		\item $f(t,\cdot)\in C^2(\mathcal{P}_2(\mathbb{R}^d))$, for all $t\in[0,T]$;
		\item $f(\cdot,\mu)\in C^1([0,T])$, for all $\mu\in \mathcal{P}_2(\mathbb{R}^d)$;
		\item $\partial_{t}f, \partial_{\mu}f, \partial_{x}\partial_{\mu}f$ and $\partial^2_{\mu}f$ are continuous.
	\end{enumerate}
\end{definition}
\begin{definition}
	We denote by $C^{1,2}_{b}([0,T]\times\mathcal{P}_2(\mathbb{R}^d))$ the set of continuous functions $f\in C^{1,2}([0,T]\times\mathcal{P}_2(\mathbb{R}^d))$ such that  the derivatives $\partial_{t}f, \partial_{\mu}f, \partial_{x}\partial_{\mu}f$ and $\partial^2_{\mu}f$ are bounded.
\end{definition}
\begin{definition}
	We denote by $C^{1,2}_{b,loc}([0,T]\times\mathcal{P}_2(\mathbb{R}^d))$ the set of continuous functions $f\in C^{1,2}([0,T]\times\mathcal{P}_2(\mathbb{R}^d))$ such that 
    the derivatives $\partial_{t}f, \partial_{\mu}f, \partial_{x}\partial_{\mu}f$ and $\partial^2_{\mu}f$ are locally bounded in $\mu$ uniformly with respect to $(t,x,y)$.

\end{definition}

\begin{remark}\label{Cdefine}
	We can define
	 $C^{1,2,2}([0,T]\times\mathbb{R}^d\times\mathcal{P}_2\left(\mathbb{R}^d\right))$, $C^{1,3}_{b,loc}([0,T]\times\mathcal{P}_2(\mathbb{R}^d))$, $\mathbb{C}_{b,loc}^{1,4}([0,T]\times\mathscr{P}_2\left(\mathbb{R}^d\right))$
 according to above procedure. In addition, it is crucial to emphasize that if $f$ belongs to $C^{2}(\mathcal{P}_2(\mathbb{R}^d))$ or $ C^{1,2}([0,T]\times\mathcal{P}_2(\mathbb{R}^d))$,  we denote by `` $\partial_{x}$'' the partial derivative w.r.t. the new variable introduced through L-derivative w.r.t. measure $\mu$. However, if $f\in \mathbb{C}^{2,2}(\mathbb{R}^d \times \mathscr{P}_2(\mathbb{R}^d) )$, we denote by `` $\partial_{x}$'' the partial derivative w.r.t. the intrinsic variable $x$ of $f$, and denote by ``$\partial y$'' the partial derivative w.r.t the new variable introduced through L-derivative.
\end{remark} 
We also recall It\^o's formula along flow of probability measures for the following It\^o process
$$d X_t=b_t d t+\sigma_t d B_t, \qquad X_0 \in L^2(\mathcal{G} ; \mathbb{R}^d),$$
where $b$ and $\sigma$ are $\mathcal{F}_t$-progressively measurable processes valued in $\mathbb{R}^d$ and $\mathbb{R}^{d\times m}$ respectively, such that
$\mathbb{E}\left[\int_{0}^{T}\left(|b_t|^2+|\sigma_t|^4 \right)dt\right]< +\infty.$
Then, if $f \in \mathbb{C}^{2}\left(\mathcal{P}_2\left(\mathbb{R}^d\right)\right)$ satisfying that
$
	\sup_{\mu\in\mathcal{K}}\left[\int_{\mathbb{R}^d}\abs{\partial_{x}\partial_{\mu}f(\mu;x)}^2\mu(dx)\right]\\<+\infty,
$ for any compact subset $\mathcal{K}\subset\mathcal{P}_2\left(\mathbb{R}^d\right)$.
We have (see Theorem 5.92 of \cite{carmona2018probabilistic}):
\begin{equation}
	\frac{d}{d t} f\left(\mathbb{P}_{X_t}\right)=\mathbb{E}\big[b_t \cdot \partial_\mu f\left(\mathbb{P}_{X_t};X_t\right)+\frac{1}{2} \operatorname{tr}\big(\sigma_t \sigma_t^{\top} \partial_x \partial_\mu f\left(\mathbb{P}_{X_t};X_t\right)\big)\big].
\end{equation}
More generally, if function $f\in \mathbb{C}^{1,2,2}\left([0,	T]\times\mathbb{R}^{d} \times\mathscr{P}_2(\mathbb{R}^{d}) \right)$, and $\partial_{\mu}f,\partial_{y}\partial_{\mu}f$ are locally bounded satisfying that for any compact subset $\mathcal{K}\subset\mathbb{R}^d\times\mathcal{P}_2\left(\mathbb{R}^d\right)$, it holds
\begin{equation*}
	\sup_{(t,x,\mu)\in[0,T]\times\mathcal{K}}\left[\int_{\mathbb{R}^d}\abs{\partial_{\mu}f(t,x,\mu;y)}^2+\abs{\partial_{y}\partial_{\mu}f(t,x,\mu;y)}^2\mu(dy)\right]<+\infty,
\end{equation*}
and if $(\xi_t)_{t\in[0,T]}$ is another It\^o process with the dynamic that 
$$d \xi_t=\eta_t d t+\gamma_t d B_t, \qquad \xi_0 \in L^2(\mathcal{G} ; \mathbb{R}^d),$$
where $\eta$ and $\gamma$ are $\mathcal{F}_t$-progressively measurable processes and valued in $\mathbb{R}^d$ and $\mathbb{R}^{d\times m}$ respectively, such that
$\mathbb{P}\left[\int_{0}^{T}\left(|\eta_t|+|\gamma_t|^2 \right)dt <\infty \right] =1.$
Then we can get (see Theorem 5.102 of \cite{carmona2018probabilistic}):
\begin{equation}
	\begin{aligned}
		&df\left(t,\xi_t,\mathbb{P}_{X_t}\right)\\
		=&\partial_{t}f\left(t,\xi_t,\mathbb{P}_{X_t}\right)dt+\partial_{x}f\left(t,\xi_t,\mathbb{P}_{X_t}\right)\cdot\eta_tdt +\dfrac{1}{2}tr\big(\partial_{x}^2 f\left(t,\xi_t,\mathbb{P}_{X_t}\right)\gamma_t \gamma_t^{\top}\big)dt+\partial_x f\left(t,\xi_t,\mathbb{P}_{X_t}\right)\cdot\gamma_tdW_t\\
		&+\widetilde{\mathbb{E}}\left[\partial_{\mu}f\big(t,\xi_t,\mathbb{P}_{X_t};\widetilde{X}_t\big)\cdot \widetilde{b}_t\right]dt+\dfrac{1}{2}\widetilde{\mathbb{E}}\left[tr\big(\widetilde{\sigma}\widetilde{\sigma}^{\top}\partial_{y}\partial_{\mu}f\big(t,\xi_t,\mathbb{P}_{X_t};\widetilde{X}_t\big)\big)\right]dt. \label{extension ito:1}
	\end{aligned}
\end{equation} 
It is noteworthy that the process $(\widetilde{X},\widetilde{b},\widetilde{\sigma})$ in (\ref{extension ito:1}) is a copy of the process $(X,b,\sigma) $ on the probability space ($\widetilde{\Omega},\widetilde{\mathcal{F}},\widetilde{\mathbb{P}}$) which is an independent copy of the probability space $(\Omega,\mathcal{F},\mathbb{P})$. More precisely, for any random variable $\theta$ on $(\Omega, \mathcal{F}, \mathbb{P})$, we denote by $\widetilde{\theta}$ an  independent copy defined on $(\widetilde{\Omega}, \widetilde{\mathcal{F}}, \widetilde{\mathbb{P}})$
that has the same law as $\theta$, i.e., $\widetilde{\mathbb{P}}_{\widetilde{\theta}}=\mathbb{P}_{\theta}$.
Throughout the paper, we consider product probability space $(\Omega \times \widetilde{\Omega} \times \overline{\Omega}, \mathcal{F} \otimes \widetilde{\mathcal{F}}\otimes \overline{\mathcal{F} }, \mathbb{P}\otimes \widetilde{\mathbb{P}}\otimes \overline{\mathbb{P} }) $, where $(\overline{\Omega},\overline{\mathcal{F}},\overline{\mathbb{P}})$ is another independent copy of the probability space $(\Omega,\mathcal{F},\mathbb{P})$ which is also an independent copy of ($\widetilde{\Omega},\widetilde{\mathcal{F}},\widetilde{\mathbb{P}}$), and $\overline{\theta}$ is another  independent copy defined on $(\overline{\Omega}, \overline{\mathcal{F}}, \overline{\mathbb{P}})$
satisfying $\overline{\mathbb{P}}_{\overline{\theta}} =\widetilde{\mathbb{P}}_{\widetilde{\theta}}=\mathbb{P}_{\theta}$.\par
The next goal is to reveal whether the L-derivative has the same partial-derivative commutativity properties as the ordinary derivative. Note that one can easily extend Lemma \ref{lem:1} and Lemma \ref{lem:2} to the multi-dimensional case.
\begin{lemma}\label{lem:1}
	If $f\in \mathbb{C}_{b,loc}^{2,2}\left(\mathbb{R}\times\mathcal{P}_2\left(\mathbb{R}\right)\right)$ ,
	then for any $x \in \mathbb{R}$, $\theta_0$$\in$ $L^{2}(\mathcal{G};\mathbb{R})$, we obtain
	\begin{equation*}
		\partial_{x}\partial_{\mu}f(x,\mathbb{P}_{\theta_0};\widetilde{\theta_0})=\partial_{\mu}\partial_{x}f(x,\mathbb{P}_{\theta_0};\widetilde{\theta_0}),\quad \widetilde{\mathbb{P} }\mbox{-}a.s.
	\end{equation*}
\end{lemma}
\begin{proof}
	For any $\gamma,\epsilon \in \mathbb{R}$, $\xi\in L^{\infty}(\mathcal{G};\mathbb{R})$, we define  $ g(x)=f(x,\mathbb{P}_{\theta_0+\gamma\xi})-f(x,\mathbb{P}_{\theta_0}) $ and
	$G(\mathbb{P}_{\theta_0})=f(x+\epsilon,\mathbb{P}_{\theta_0})-f(x,\mathbb{P}_{\theta_0}),$
	then we have 
	\begin{equation}\label{gG}
		g(x+\epsilon)-g(x)=G(\mathbb{P}_{\theta_0+\gamma\xi})-G(\mathbb{P}_{\theta_0}).
	\end{equation}
	Due to the fact that
	\begin{equation*}
		\begin{aligned}
			g(x+\epsilon)-g(x)&=\int_{0}^{\epsilon}\partial_{x}g(x+\lambda)d\lambda\\
			&=\int_{0}^{\epsilon}\int_{0}^{\gamma}\widetilde{\mathbb{E}}\left[\partial_{\mu}\partial_{x}f(x+\lambda,\mathbb{P}_{\theta_0+\lambda^{'}\xi};\widetilde{\theta_0}+\lambda^{'}\widetilde{\xi})\widetilde{\xi}\right]d\lambda^{'} d\lambda,
		\end{aligned}
	\end{equation*}
	and
	\begin{equation*}
		\begin{aligned}
			G(\mathbb{P}_{\theta_0+\gamma\xi})-G(\mathbb{P}_{\theta_0})&=\int_{0}^{\gamma}\widetilde{\mathbb{E}}\left[\partial_{\mu}G(\mathbb{P}_{\theta_0+\lambda^{'}\xi};\widetilde{\theta_0}+\lambda^{'}\widetilde{\xi})\widetilde{\xi}\right]d\lambda^{'}\\
			&=\int_{0}^{\gamma}\int_{0}^{\epsilon}\widetilde{\mathbb{E}}\left[\partial_{x}\partial_{\mu}f(x+\lambda,\mathbb{P}_{\theta_0+\lambda^{'}\xi};\widetilde{\theta_0}+\lambda^{'}\widetilde{\xi})\widetilde{\xi}\right])d\lambda d\lambda^{'}.
		\end{aligned}
	\end{equation*}
	We divide both sides of (\ref{gG}) by $\gamma\epsilon$ and let $\gamma\epsilon$ $\rightarrow$ 0. According to the Lebesgue differentiation theorem we obtain
	$
	\widetilde{\mathbb{E}}\left[\partial_{x}\partial_{\mu}f(x,\mathbb{P}_{\theta_0};\widetilde{\theta_0})\widetilde{\xi}\right]=\widetilde{\mathbb{E}}\left[\partial_{\mu}\partial_{x}f(x,\mathbb{P}_{\theta_0};\widetilde{\theta_0})\widetilde{\xi}\right], \text{ for any $\widetilde{\xi}\in L^{\infty}(\mathcal{\widetilde{G}};\mathbb{R})$}.
	$
	Considering the arbitrariness of $\widetilde{\xi}\in L^{\infty}(\widetilde{\mathcal{G}};\mathbb{R})$, we have
	$\partial_{x}\partial_{\mu}f(x,\mathbb{P}_{\theta_0};\widetilde{\theta_0})=\partial_{\mu}\partial_{x}f(x,\mathbb{P}_{\theta_0};\widetilde{\theta_0}),$ $\widetilde{\mathbb{P} }\mbox{-}a.s.$
\end{proof}
Similarly, we can also get 
\begin{lemma}\label{lem:2}
	If $f\in \mathbb{C}_{b,loc}^{1,2}\left([0,T]\times\mathcal{P}_2\left(\mathbb{R}\right)\right)$, and $\partial_{\mu}\partial_{t}f$, $\partial_{t}\partial_{\mu}f$ are continuous and bounded. Then for any $t\in[0,T]$, $\theta_0\in L^{2}(\mathcal{G};\mathbb{R})$, we have
$
		\partial_{t}\partial_{\mu}f(t,\mathbb{P}_{\theta_0}; \theta_0)=\partial_{\mu}\partial_{t}f(t,\mathbb{P}_{\theta_0};\theta_0), \mathbb{P}\mbox{-}a.s.
$
\end{lemma}
\begin{lemma}[see Corollary 5.89 of \cite{carmona2018probabilistic}]\label{lem:3}
	Assume that $f \in \mathbb{C}^{2}\left(\mathcal{P}_2(\mathbb{R}^{d})\right)$. Then for any $\mu\in\mathcal{P}_{2}\left(\mathbb{R}^d\right)$ and any $x,y\in\mathbb{R}^{d}$, $\partial_{x}$$\partial_{\mu}$$f$ is a symmetric matrix and $\partial_{\mu}^{2}$$f$$(\mu;x,y)$=$\partial_{\mu}^{2}$$f$$(\mu;y,x)^{\top}$.
	
\end{lemma} \par

\subsection{Problem Statement}
Now we begin to state our control problem and assumptions.
 We denote by $\mathcal{U}$ the set of control processes, namely the family of all $\mathbb{F}$-progressively measurable processes $\mathbf{u}:[0, T] \times \Omega \rightarrow U$, where $U$ is a Polish space.	
For every $s\in \left [ 0,T \right ]$, $\xi \in L^{2} \left ( \mathcal{F} _{s};\mathbb{R}^{d} \right )$, $\mathbf{u}  \in \mathcal{U}$, the state process evolves according to the following controlled McKean-Vlasov stochastic differential equation
\begin{equation}
	X^{s,\xi,\mathbf{u}}_{t}=\xi+\int_{s}^{t} b\big(r,X^{s,\xi,\mathbf{u}}_{r},\mathbb{P}_{X^{s,\xi,\mathbf{u}}_{r}},\mathbf{u}_{r}\big)dt+\int_{s}^{t} \sigma\big(r,X^{s,\xi,\mathbf{u}}_{r},\mathbb{P}_{X^{s,\xi,\mathbf{u}}_{r}},\mathbf{u}_{r}\big)dB_t, \quad t\in \left[s,T\right],\label{ques:1}
\end{equation}
\noindent with the reward functional given by
\begin{equation}
	J\left(s,\xi,\mathbf{u}\right)=\mathbb{E}\left[\int_{s}^{T} f\big(r,X_{r}^{s,\xi,\mathbf{u}},\mathbb{P}_{X_{r}^{s,\xi,\mathbf{u}}},\mathbf{u}_{r}\big)dr+
	h\big(X_{T}^{s,\xi,\mathbf{u}},\mathbb{P}_{X_{T}^{s,\xi,\mathbf{u}}}\big)\right],\label{ques:2}
\end{equation}
where the coefficients $b:[0, T] \times \mathbb{R}^d \times$ $\mathcal{P}_2\left(\mathbb{R}^d\right) \times U\rightarrow \mathbb{R}^d$, $\sigma:[0, T] \times \mathbb{R}^d \times U \rightarrow \mathbb{R}^{d \times m}$, $f:[0, T] \times \mathbb{R}^d \times \mathcal{P}_2\left(\mathbb{R}^d\right) \times U \rightarrow \mathbb{R}$, and $h: \mathbb{R}^d \times \mathcal{P}_2\left(\mathbb{R}^d\right) \rightarrow \mathbb{R}$ satisfy the following assumptions and we denote $b=(b^1,\dots,b^d)^{\top}$, $\sigma=(\sigma^1,\dots,\sigma^m)=[\sigma^{i,j}]_{1\le i\le d ,1\le j \le m}$.
\begin{assumption}
	The coefficients $b, \sigma, f, h$ are  jointly measurable, uniformly continuous and uniformly bounded in all variables. For all $(t,u) \in [0,T]\times U$, we have $b^i(t,\cdot, \cdot, u)$, $\sigma^{i,j}(t,\cdot, \cdot, u)$, $f(t,\cdot, \cdot, u)$, $h(\cdot, \cdot) \in \mathbb{C}^{1,1}(\mathbb{R}^d \times \mathcal{P}_2(\mathbb{R}^d) )$.
	Moreover, all the first-order derivatives are uniformly bounded and Lipschitz
	continuous in all variables.
	\label{ass:3}
\end{assumption}
\begin{assumption}
	The coefficients $b, \sigma, f, h$ satisfy Assumption \ref{ass:3}. For all $(t,u) \in [0,T]\times U$, it follows that $b^i(t,\cdot, \cdot, u), \sigma^{i,j}(t,\cdot, \cdot, u), f(t,\cdot, \cdot, u),h(\cdot, \cdot) \in \mathbb{C}^{2,2}\left(\mathbb{R}^d \times\mathscr{P}_2(\mathbb{R}^d)\right)$. Furthermore, all the second-order derivatives are uniformly bounded and Lipschitz
	continuous in all variables.
	\label{ass:4}
\end{assumption}
Our optimal control problem can be stated as follows.\par
\noindent\textbf{Problem ($\mathbf{M_{s\xi}}$).} For a given $(s,\xi)\in [0,T]\times L^{2} \left ( \mathcal{F} _{s} ;\mathbb{R}^{d} \right )$, minimize (\ref{ques:2}) over $\mathcal{U}$ subject to the state equation (\ref{ques:1}).\par
We introduce the function V, to which we refer as the lifted value function, defined as
\begin{equation}\label{lifted v}
	V\left(s,\xi\right)=\inf_{\mathbf{u} \in \mathcal{U}}J\left(s,\xi,\mathbf{u}\right), \quad\quad 
	\forall\left(s,\xi\right) \in [0,T]\times L^{2}  ( \mathcal{F} _{s} ;\mathbb{R}^{d}  ).
\end{equation}
Any $\mathbf{u}^{\ast}\in \mathcal{U}$ satisfying 
$
	J\left(s,\xi,\mathbf{u}^{\ast}\right)=\inf\limits_{\mathbf{u} \in \mathcal{U}}J\left(s,\xi,\mathbf{u})\right)
$ is called an optimal control, and the corresponding $X^{\ast}\equiv X^{s,\xi,\mathbf{u}^{\ast}}$ is called optimal state and $(X^{\ast},\mathbf{u}^{\ast})$ is called an optimal pair.\par

\section{MP and DPP for Mean-Field Optimal Control Problems}
\label{section3}
In this section, we first present the maximum principle for the mean-field optimal control problem, as given by \cite{buckdahn2016stochastic}. We mention that we use the modified form of the second-order adjoint equation proposed in \cite{Nie-Yan-Mean-field-partial-information-non-zero-sum}. Then, we give the dynamic programming principle and the corresponding HJB equation, as shown in \cite{cosso2023optimal}. Finally, we present a general verification theorem.
\subsection{Stochastic Maximum Principle} \label{subsection3.1}\label{sec:3.1}
 We introduce the first-order and second-order adjoint equations before presenting the maximum principle. Compared with \cite{buckdahn2016stochastic}, we present the specific forms of MP in multi-dimensional case without proof. Now we introduce some simplified notations. Let $\phi$ be a generic function representing $f$, $h$, $[b^{i}]$, $[\sigma^{i,j}]$, and let $\left(X^*,\mathbf{u}^*\right)$ be an optimal pair. We denote\par
 
 \begin{equation*}\label{dao:1}
 	\left\{
 	\begin{aligned}
 			&\phi(t)=\phi(t,X_t^{\ast},\mathbb{P}_{X_t^{\ast}},\mathbf{u}_t^{\ast}),\\
 		&\widetilde{\phi}(t)=\phi(t,\widetilde{X}_t^{\ast},\mathbb{P}_{X_t^{\ast}},\widetilde{\mathbf{u}}_t^{\ast}),\\
 		&\phi_x(t)=\partial_x \phi\left(t,X_t^*, \mathbb{P}_{X_t^*}, \mathbf{u}_t^* \right) ,
 	\end{aligned}
 	\right.
 	\quad
 	\left\{
 	\begin{aligned}
 		&\tilde{\phi}_x(t)=\partial_x \phi(t,\widetilde{X}_t^*, \mathbb{P}_{X_t^*}, \widetilde{\mathbf{u}}_t^*),\\
 		&\tilde{\phi}_\mu(t)=\partial_\mu \phi(t,X_t^*, \mathbb{P}_{X_t^*}, \mathbf{u}_t^* ; \widetilde{X}_t^*) ,\\
 		& \tilde{\phi}_\mu^*(t)=\partial_\mu \phi(t,\widetilde{X}_t^*, \mathbb{P}_{X_t^*}, \widetilde{\mathbf{u}}_t^* ; X_t^*),
 	\end{aligned}
 	\right.
 \end{equation*}
 and for $\psi=b$, $[\sigma^j]$, \begin{equation*}\label{dao:2}
 	\left\{
 	\begin{aligned}
 		&\psi_x(t)=([\psi^1]_x(t),\cdots,[\psi^d]_x(t))^{\top} ,\\
 		&\tilde{\psi}_x(t)=([\widetilde{\psi^1}]_x(t),\cdots,[\widetilde{\psi^d}]_x(t))^{\top},
 	\end{aligned}
 	\right.
 	\quad
 	\left\{
 	\begin{aligned}
 		&\tilde{\psi}_\mu(t)=([\widetilde{\psi^1}]_{\mu}(t),\cdots,[\widetilde{\psi^d}]_{\mu}(t))^{\top} ,\\
 		& \tilde{\psi}_\mu^*(t)=([\widetilde{\psi^1}]^{\ast}_{\mu}(t),\cdots,[\widetilde{\psi^d}]^{\ast}_{\mu}(t))^{\top}.
 	\end{aligned}
 	\right.
 \end{equation*} We denote their second-order derivatives by
 \begin{equation*}
 	\left\{
 	\begin{aligned}
 		&\phi_{xx}(t)=\partial_x^2 \phi\big(t,X_t^*, \mathbb{P}_{X_t^*}, \mathbf{u}_t^* \big), \\
 		&\widetilde{\phi}_{xx}(t)=\partial_x^2 \phi\big(t,\widetilde{X}_t^*, \mathbb{P}_{X_t^*}, \widetilde{\mathbf{u}}_t^* \big),\\
 		&\widetilde{\phi}_{x\mu}(t)=\partial_x\partial_{\mu} \phi(t,X_t^*, \mathbb{P}_{X_t^*}, \mathbf{u}_t^* ;\widetilde{X}_t^*),\\
 		&\widetilde{\phi}^{\ast}_{x\mu}(t)=\partial_x\partial_{\mu} \phi(t,\widetilde{X}_t^*, \mathbb{P}_{X_t^*}, \widetilde{\mathbf{u}}_t^* ;X_t^*),
 	\end{aligned}
 	\right.
 	\quad
 	\left\{
 	\begin{aligned}
 		&\widetilde{\phi}_{y \mu}(t)=\partial_y \partial_\mu \phi\left(t,X_t^*, \mathbb{P}_{X_t^*}, \mathbf{u}_t^* ; \widetilde{X}_t^*\right),\\
 		&\widetilde{\phi}_{y \mu}^*(t)=\partial_y \partial_\mu \phi\left(t,\widetilde{X}_t^*, \mathbb{P}_{X_t^*}, \widetilde{\mathbf{u}}_t^* ; X_t^*\right),\\
 		&\overline{\phi}_{\mu\mu}(t)=\partial_{\mu}^2 \phi(t,X_t^*, \mathbb{P}_{X_t^*}, \mathbf{u}_t^* ;\widetilde{X}_t^*,\overline{X}_t^{\ast}).
 	\end{aligned}
 	\right.
 \end{equation*}
Consider the first-order adjoint equation, which is the following mean-field type linear
BSDE:
\begin{equation}
	\left\{
	\begin{aligned}
		dp_{t} =&-\bigg\{b_{x}\left(t\right)^{\top} p_{t}+\widetilde{\mathbb{E}}\left[\widetilde{b_{\mu}^{\ast}}\left(t\right)^{\top} \widetilde{p}_{t}\right]+\sum_{j=1}^{m}[\sigma^{j}]_{x}\left(t\right)^{\top} q^j_{t}+\sum_{j=1}^{m}\widetilde{\mathbb{E}}\left[[\widetilde{\sigma^{j}}]^{\ast}_{\mu}\left(t\right)^{\top}\widetilde{q}^j_{t}\right] \\
		&
		-f_{x}\left(t\right)-\widetilde{\mathbb{E}}\left[\widetilde{f}_{\mu}^{\ast}\left(t\right)\right]\bigg\} dt+q_{t}dB_{t},
		\\
		p_{T}=&-h_{x}\left(T\right)-\widetilde{\mathbb{E}}\left[\widetilde{h}_{\mu}^{\ast}\left(T\right)\right],\label{bansui:1}
	\end{aligned}
	\right.
\end{equation}
We know that according to Theorem 3.1 of \cite{buckdahn2009mean}, under Assumption \ref{ass:3}, BSDE (\ref{bansui:1}) admits a unique $\mathbb{F}$-adapted solution $(p, q)$ which takes values in $ (\mathbb{R}^{d},\mathbb{R}^{d\times m})$ and satisfies
$
	\mathbb{E}\Big[\sup\limits_{t\in[0,T]}\abs{p_t}^2+\int_{0}^{T}\abs{q_t}^2dt\Big]<\infty.
$

Let us define the following  Hamiltonian function $H:[0,T]\times\mathbb{R}^d \times\mathcal{P}_{2}(\mathbb{R}^d)\times U\times\mathbb{R}^{d}\times\mathbb{R}^{d\times m}$ as
\begin{equation*}\label{Hdefine}
	H\left ( t,x,\mu ,u,p,q \right ) := b\left ( t,x,\mu ,u \right )\cdot  p+  tr\left(q^{T}\sigma \left ( t,x,\mu,u\right )\right)-f\left (t, x,\mu,u \right ),
\end{equation*}
and denote, for $t \in[s, T]$,
\begin{equation}\label{dao:3}
	\left\{\begin{array}{l}
		H_{x x}(t):=\sum_{i=1}^{d}[b^i]_{x x}(t) p^i_t+\sum_{j=1}^{m}\sum_{i=1}^{d}[\sigma^{i,j}]_{x x}(t)  q^{i,j}_t-f_{x x}(t) ,\\
		\widetilde{H}^{\ast}_{y \mu}(t):=\sum_{i=1}^{d}[\widetilde{b^i}]^{\ast}_{y \mu}(t) \widetilde{p}^i_t+\sum_{j=1}^{m}\sum_{i=1}^{d}[\widetilde{\sigma}^{i,j}]^{\ast}_{y \mu}(t) \widetilde{q}^{i,j}_t-\widetilde{f}^{\ast}_{y \mu}(t) .
	\end{array} \right.
\end{equation}
Then we consider the following second-order adjoint equation
\begin{equation}
	\left\{\begin{aligned}
		d P_t= & -\bigg\{\left(b_x(t)\right)^{\top} P_t+P_t b_x(t)+\sum_{j=1}^{m}[\sigma^{j}]_{x}^{\top}(t) P_t[\sigma^{j}]_{x}(t)+\sum_{j=1}^{m}\left([\sigma^{j}]_x(t)\right)^{\top} Q_t^{j} \\
		& +\sum_{j=1}^{m}Q_t^{j}[\sigma^{j}]_x(t) +\left(H_{x x}(t)+\widetilde{\mathbb{E}}\left[\widetilde{H}_{ y\mu}^*(t)\right]\right)\bigg\} d t+\sum_{j=1}^{m}Q^{j}_t d B^j_t,\\
		P_T= &- h_{x x}(T)-\widetilde{\mathbb{E}}\left[\widetilde{h}_{y\mu}^*(T)\right].
	\end{aligned}\right.\label{bansui:2}
\end{equation}
Note that (\ref{bansui:2}) is a standard linear BSDE. It is well-known that it has a unique
$\mathbb{F}$-adapted solution $(P_t, Q_t)$ satisfying
$
	\mathbb{E}\Big[\sup\limits_{t\in[0,T]}\abs{P_t}^2+\sum_{j=1}^{m}\int_{0}^{T}|Q^j_t|^2dt\Big]<\infty.
$

\begin{remark}
	The second-order adjoint equation (\ref{bansui:2}) is different from the one in paper \cite{buckdahn2016stochastic}, specifically equation (\ref{bansui:2}) does not contain terms involving $\widetilde{\mathbb{E}}\big[\widetilde{b}_{\mu}^*(t)\big]$ and $\widetilde{\mathbb{E}}\big[\widetilde{\sigma}_{\mu}^*(t)\big]$. For more details, interested readers can refer to Remark 2.3 of \cite{Nie-Yan-Mean-field-partial-information-non-zero-sum} .
\end{remark}
At the end of this subsection, we recall the stochastic maximum principle given in \cite{buckdahn2016stochastic}, where we use the modified second-order adjoint equation (see \cite{Nie-Yan-Mean-field-partial-information-non-zero-sum}).
\begin{theorem}\label{MP}[Stochastic Maximum Principle]. Suppose that Assumption \ref{ass:4} holds. Let $( X^{\ast},\mathbf{u}^{\ast})$ be an optimal pair of the \textbf{Problem} $\bm{(}\bm{M_{s\xi}}\bm{)}$,
	then there are two pairs of $\mathbb{F}$-adapted processes $(p, q)$ and $(P, Q)$ which solve (\ref{bansui:1}) and (\ref{bansui:2}) respectively, such that for all u $\in$U and a.e. $t \in[s, T]$, it holds $\mathbb{P}$-a.s. that\\
	\begin{equation}
		\begin{aligned}
			\begin{aligned}
				&H\left(t,X_t^*, \mathbb{P}_{X_t^*}, u, p_t, q_t\right)-H\left(t,X_t^*, \mathbb{P}_{X_t^*}, \mathbf{u}_t^*, p_t, q_t\right)+\frac{1}{2}\left[\Delta \sigma^{*, u}(t, \cdot)\right]^T P_t\left[\Delta \sigma^{*, u}(t, \cdot)\right] \leq 0 ,
			\end{aligned}
		\end{aligned}	   
	\end{equation}
	where $\Delta \sigma^{*, u}(t, \cdot):=\sigma\left(t, X_t^*, \mathbb{P}_{X_t^*}, u\right)-\sigma\left(t, X_t^*, \mathbb{P}_{X_t^*}, \mathbf{u}_t^*\right)$.
\end{theorem}

\subsection{Dynamic Programming Principle}
Regarding the DPP of mean-field optimal control problem, the method currently acknowledged is to define the value function on the space of probability measures. The difficulty of this method lies in proving that lifted value function $V$ given by (\ref{lifted v}) is law-invariant, i.e., depends only on the distribution of $\xi$. It should be noted that this conclusion is not trivial. In the case where the control has a feedback form, one can refer to \cite{pham2018bellman} for the law invariance property. For the general system considered in this paper, the following conclusion can be referred to Theorem 3.6 of \cite{cosso2023optimal}.
\begin{theorem}[Law invariance property]\label{law:1}
	Suppose that Assumption \ref{ass:4} holds. Then, the map V satisfies the law invariance property: for every $s$ $\in$ $\left[0,T\right]$ and $\xi,\xi^{'}\in L^{2} \left ( \mathcal{F} _{s} ;\mathbb{R}^{d} \right )$, with $\mathbb{P}_{\xi}=\mathbb{P}_{\xi^{'}}$, it holds that $V(s,\xi)$=$V(s,\xi^{'}).$
\end{theorem}
\noindent As a consequence of  Theorem \ref{law:1}, we can define the
value function $v:\left[0,T\right]\times\mathcal{P}_{2}\left(\mathbb{R}^{d}\right)$$\to\mathbb{R}$ as
$
	v\left(s,\mu\right)=V\left(s,\xi\right)$, for any $\left(s,\mu\right)\in\left[0,T\right]\times\mathcal{P}_{2}(\mathbb{R}^d),
$ with any $\xi\in L^{2} \left ( \mathcal{F}_s;\mathbb{R}^{d} \right )$ satisfying $\mathbb{P}_{\xi}=\mu$.	Then we have
\begin{theorem}[Corollary 3.9 of \cite{cosso2023optimal}]\label{DDP:1}
	Suppose that Assumption \ref{ass:4} holds, then v satisfies the dynamic programming principle: for all $s,t \in\left[0,T\right]$, with $s\le t$, it holds that
	\begin{equation}
		v\left(s,\mu\right)=\inf\limits_{\mathbf{u} \in \mathcal{U}}\left\{\mathbb{E}\left[\int_{s}^{t}f\big(r,X_{r}^{s,\xi,\mathbf{u}},\mathbb{P}_{X_{r}^{s,\xi,\mathbf{u}}},\mathbf{u}_{r}\big)dr\right]+v\big(t,\mathbb{P}_{X_{t}^{s,\xi,\mathbf{u}}}\big)\right\},\label{th3.3:1}
	\end{equation} \label{th3.3}
	for any $\xi\in L^{2} \left ( \mathcal{F} _{s};\mathbb{R}^{d} \right )$, with $\mathbb{P}_\xi =\mu$.
\end{theorem}
Now we give the HJB equation which is deduced from the DPP. 
For every t $\in$ $\left[0,T\right]$, we introduce the set
$
\mathcal{M}_{t}:=\left\{\mathcal{u} :\Omega\to U:\mathcal{u}\text{ is }\mathcal{F}_{t}\text{-measurable} \right\}.
$
Note that, we have $\mathcal{M}_t=\bigcap_{\varepsilon>0} \mathcal{M}_{t+\varepsilon}$ since the filtration $\mathcal{F}_t$ is right-continuous. Then, we consider the second-order partial differential equation on $\left[0,T\right]$$\times$$\mathcal{P}_{2}$$\left(\mathbb{R}^{d}\right)$ (see equation (5.1) in \cite{cosso2023optimal}):
\begin{equation}
	\left\{
	\begin{aligned}
		&	-\partial_{t}w\left(t,\mu\right) +\sup\limits_{\mathcal{u} \in \mathcal{M}_{t}}\displaystyle\mathbb{E}\bigg[
		-f\left(t,\xi,\mu,\mathcal{u} \right)- b\left(t,\xi,\mu,\mathcal{u} \right)\cdot\partial_{\mu}w\left(t,\mu;\xi \right) \\
		&\left.\qquad\qquad\quad\hspace{0.15em}-\dfrac{1}{2}tr\big(\sigma\sigma^{\textsuperscript {T}}\left(t,\xi ,\mu,\mathcal{u} \right)\partial_{x}\partial_{\mu}w\left(t,\mu;\xi\right)\big)\right]=0, \\
		&w\left(T,\mu\right)=\int_{\mathbb{R}^{d}}h\left(x,\mu\right)\mu\left(dx\right),
	\end{aligned}\label{hjb:1}
	\right.
\end{equation}
with any $\xi\in L^2(\mathcal{G},\mathbb{R}^d)$, such that $\mathbb{P}_{\xi}=\mu.$ \par
We define the generalized Hamiltonian function $G:[0,T]\times\mathbb{R}^d \times\mathcal{P}_{2}(\mathbb{R}^d)\times U\times\mathbb{R}^{d}\times\mathbb{R}^{d\times d}$ as
\begin{equation}\label{Gdefine}
 	G\left ( t,x,\mu ,u,p,P \right ) := b\left ( t,x,\mu ,u \right )\cdot  p+  \frac{1}{2}tr\left(P\sigma \left ( t,x,\mu,u\right )\sigma \left ( t,x,\mu,u\right )^{\top}\right)-f\left (t, x,\mu,u \right ),
 \end{equation} then equation (\ref{hjb:1}) can be rewritten as 
\begin{equation*}
	\left\{
	\begin{aligned}
		&	-\partial_{t}w\left(t,\mu\right) +\sup\limits_{\mathcal{u} \in \mathcal{M}_{t}}
		\mathbb{E}\left[G\left ( t,\xi,\mu ,\mathcal{u},-\partial_{\mu}w\left(t,\mu;\xi\right),-\partial_{x}\partial_{\mu}w\left(t,\mu;\xi\right)\right ) \right]
		=0 ,\\
		&w\left(T,\mu\right)=\mathbb{E}\left[h(\xi,\mu)\right].
	\end{aligned}
	\right.
\end{equation*}
\begin{definition}
	A function $w: [0, T]\times \mathcal{P}_2(\mathbb{R}^d) \mapsto \mathbb{R}$ is a classical solution to equation
	(\ref{hjb:1}) if $w\in C_{b,loc}^{1,2}\left([0,T]\times\mathcal{P}_2\left(\mathbb{R}^d\right)\right)$ and satisfies (\ref{hjb:1}).
\end{definition}
Following Theorem 5.3 of \cite{cosso2023optimal}, we have
\begin{theorem}\label{classical}
	Let Assumption \ref{ass:4} hold and assume that the value function $v$  belongs to the space $C_{b,loc}^{1,2}\left([0,T]\times\mathcal{P}_2\left(\mathbb{R}^d\right)\right)$. Then $v$ is a classical solution of equation (\ref{hjb:1}).
	\begin{proof}
	 Theorem 5.3 in \cite{cosso2023optimal} requires that 
		$v\in C_{b}^{1,2}\left([0,T]\times\mathcal{P}_2\left(\mathbb{R}^d\right)\right)$ to ensure the applicability of Itô's formula. Recalling the It\^o's formula mentioned in Section 2.1, we note that it remains applicable when $v\in C_{b,loc}^{1,2}\left([0,T]\times\mathcal{P}_2\left(\mathbb{R}^d\right)\right)$. Thus, proceeding along the same lines as in the proof of Theorem 5.3 in \cite{cosso2023optimal}, we can obtain this result.
	\end{proof}
\end{theorem}

\subsection{Mean-Field Stochastic Verification Theorem}
The following result which is called  verification theorem gives a way of testing whether a given admissible pair is optimal or not. Notably, there is currently little literature for the mean-field stochastic verification theorem. To our knowledge, \cite{pham2017dynamic} and \cite{pham2018bellman} give the verification theorem of mean-field optimal control problem in the case with feedback control. Here we give a generalized result.\par
\begin{theorem}\label{ver}
	Let Assumption \ref{ass:4} hold and $w$ $\in$ $C_{b,loc}^{1,2}$$\left(\left[0,T\right]\times\mathcal{P}_{2}(\mathbb{R}^d)\right)$ be a solution of HJB equation (\ref{hjb:1}). Then 
	\begin{enumerate}[(i)]
		\item $	w(s,\mu)\le J(s,\xi,\mathbf{u}),$ for all $ \mathbf{u}\in  \mathcal{U},$ $(s,\mu) \in [0,T]\times\mathcal{P}(\mathbb{R}^d),$ with $\xi\in L^2(\mathcal{F}_s;\mathbb{R}^d)$, s.t. $\mathbb{P}_{\xi}=\mu\label{th4.1:1}$.

		\item \label{ver:ii}An admissible pair $(X^{\ast},\mathbf{u}^{\ast})$ is optimal for \textbf{Problem} $\bm{(}\bm{M_{s\xi}}\bm{)}$, if for $a.e.$ $  t\in[s,T]$, it holds that
		\begin{equation}
			\begin{aligned}
				\partial_{t}w\left(t,\mathbb{P}_{X_{t}^{\ast}}\right) &\equiv 	\sup\limits_{\mathcal{u} \in \mathcal{M}_{t}}\mathbb{E}\left[G\left ( t,X_t^{\ast},\mathbb{P}_{X_{t}^{\ast}} ,\mathcal{u},-\partial_{\mu}w\left(t,\mathbb{P}_{X_{t}^{\ast}};X_t^{\ast}\right),-\partial_{x}\partial_{\mu}w\left(t,\mathbb{P}_{X_{t}^{\ast}};X_t^{\ast}\right)\right ) \right]	\\
				&=\mathbb{E}\left[G\left ( t,X_t^{\ast},\mathbb{P}_{X_{t}^{\ast}} ,\mathbf{u}_{t}^{\ast},-\partial_{\mu}w\left(t,\mathbb{P}_{X_{t}^{\ast}};X_t^{\ast}\right),-\partial_{x}\partial_{\mu}w\left(t,\mathbb{P}_{X_{t}^{\ast}};X_t^{\ast}\right)\right ) \right. .
			\end{aligned}\label{ver：2}
		\end{equation}
		\item \label{ver:iii}Let  $v$ $\in$ $C_{b,loc}^{1,2}$$\left(\left[0,T\right]\times\mathcal{P}_{2}(\mathbb{R}^d)\right)$ be the value function of \textbf{Problem} $\bm{(}\bm{M_{s\xi}}\bm{)}$. Then an admissible pair $(X^{\ast},\mathbf{u}^{\ast})$ is optimal for \textbf{Problem} $\bm{(}\bm{M_{s\xi}}\bm{)}$ if and only if for $a.e.$ $t\in[s,T]$,
		\begin{equation}\label{veriii:equ}
			\begin{aligned}
				\partial_{t}v\left(t,\mathbb{P}_{X_{t}^{\ast}}\right) &\equiv 	\sup\limits_{\mathcal{u} \in \mathcal{M}_{t}}\mathbb{E}\left[G\left ( t,X_t,\mathbb{P}_{X_{t}} ,\mathcal{u},-\partial_{\mu}v\left(t,\mathbb{P}_{X_{t}};X_t\right),-\partial_{x}\partial_{\mu}v\left(t,\mathbb{P}_{X_{t}};X_t\right)\right ) \right]	\\
				&=\mathbb{E}\left[G\left ( t,X_t^{\ast},\mathbb{P}_{X_{t}^{\ast}} ,\mathbf{u}^{\ast}_t,-\partial_{\mu}v\left(t,\mathbb{P}_{X_{t}^{\ast}};X_t^{\ast}\right),-\partial_{x}\partial_{\mu}v\left(t,\mathbb{P}_{X_{t}^{\ast}};X_t^{\ast}\right)\right ) \right].
			\end{aligned}
		\end{equation}
	\end{enumerate}
\end{theorem}
	\begin{proof}
		(i) For any $\mathbf{u}\in \mathcal{U}$, let $X$ be the corresponding state trajectory. Applying It\^o's formula to $	w\left(t,\mathbb{P}_{X_{t}}\right)$, by noting the terminal condition of (\ref{hjb:1}) we get
		\begin{equation}
			\begin{aligned}
				w(s,\mu)=&J(s,\xi,\mathbf{u})+\int_{s}^{T}\left\{-\partial_{t}w(t,\mathbb{P}_{X_t})\right.\\
				&+\left.\mathbb{E}\left[G\left ( t,X_t,\mathbb{P}_{X_{t}} ,\mathbf{u}_{t},-\partial_{\mu}w\left(t,\mathbb{P}_{X_{t}};X_t\right),-\partial_{x}\partial_{\mu}w\left(t,\mathbb{P}_{X_{t}};X_t\right)\right ) \right]\right\}dt\\		
				\le&J(s,\xi,\mathbf{u})+\int_{s}^{T}\Big\{-\partial_{t}w(t,\mathbb{P}_{X_t})\\
				&+\left.	\sup\limits_{\mathcal{u} \in \mathcal{M}_{t}}\mathbb{E}\left[G\left ( t,X_t,\mathbb{P}_{X_{t}} ,\mathcal{u},-\partial_{\mu}w\left(t,\mathbb{P}_{X_{t}};X_t\right),-\partial_{x}\partial_{\mu}w\left(t,\mathbb{P}_{X_{t}};X_t\right)\right ) \right]\right\}dt\\
				=&J(s,\xi,\mathbf{u}).
			\end{aligned}\label{th4.1proof:1}
		\end{equation}
		 The last equality is due to that $w$ is a solution of HJB equation (\ref{hjb:1}). \par
		(ii) On the one hand, similar to (\ref{th4.1proof:1}), by using (\ref{ver：2}) we obtain
		\begin{equation}
			\begin{aligned}
				w(s,\mu)=&J(s,\xi,\mathbf{u}^{\ast})+\int_{s}^{T}\left\{-\partial_{t}w(t,\mathbb{P}_{X_t^{\ast}})\right.\\
				&+\left.\mathbb{E}\left[G\left ( t,X_t^{\ast},\mathbb{P}_{X_{t}^{\ast}} ,\mathbf{u}^{\ast}_t,-\partial_{\mu}w\left(t,\mathbb{P}_{X_{t}^{\ast}};X_t^{\ast}\right),-\partial_{x}\partial_{\mu}w\left(t,\mathbb{P}_{X_{t}^{\ast}};X_t^{\ast}\right)\right ) \right]\right\}dt,\\
				=&J(s,\xi,\mathbf{u}^{\ast}(\cdot)),
			\end{aligned}\label{ver:3}
		\end{equation}
		which shows $w(s,\mu)\ge v(s,\mu)$ by recalling the definition of $v$ (see the definition below Theorem \ref{law:1}). On the other hand, from (\ref{th4.1:1}) we can obtain $w(s,\mu)\le v(s,\mu)$. Thus $v(s,\mu)=w(s,\mu)=J(s,\xi,\mathbf{u}^{\ast}(\cdot))$ which means that $(X^{\ast},\mathbf{u}^{\ast})$ is optimal for \textbf{Problem ($\mathbf{M_{s\xi}}$)}. \par
		(iii) By noting Theorem \ref{classical} and (ii), we know the ``if'' part holds, thus it is only to show the ``only if'' part. In fact, if an admissible pair $(X^{\ast},\mathbf{u}^{\ast})$ is optimal, we have 
				$v(s,\mu)=J(s,\xi,\mathbf{u}^{\ast})$.  Then by applying It\^o's formula to $v(t,\mathbb{P}_{X_t^{\ast}})$ and noting that $v(T,\mathbb{P}_{X_T^{\ast}})=\mathbb{E}[h(X_T^{\ast},\mathbb{P}_{X_T^{\ast}})]$ from the definition of $v$, we obtain
				\begin{equation*}
					\begin{aligned}
						v(s,\mu)=&J(s,\xi,\mathbf{u}^{\ast})+\int_{s}^{T}\left\{-\partial_{t}v(t,\mathbb{P}_{X_t^{\ast}})\right.\\
						&+\left.\mathbb{E}\left[G\left ( t,X_t^{\ast},\mathbb{P}_{X_{t}^{\ast}} ,\mathbf{u}^{\ast}_t,-\partial_{\mu}v\left(t,\mathbb{P}_{X_{t}^{\ast}};X_t^{\ast}\right),-\partial_{x}\partial_{\mu}v\left(t,\mathbb{P}_{X_{t}^{\ast}};X_t^{\ast}\right)\right ) \right]\right\}dt.
					\end{aligned}
				\end{equation*}
		 Thus, we can get \begin{equation*}\label{yanzheng-iii}
			\int_{s}^{T}\left\{-\partial_{t}v(t,\mathbb{P}_{X_t^{\ast}})
			+\mathbb{E}\left[G\left ( t,X_t^{\ast},\mathbb{P}_{X_{t}^{\ast}} ,\mathbf{u}^{\ast}_t,-\partial_{\mu}v\left(t,\mathbb{P}_{X_{t}^{\ast}};X_t^{\ast}\right),-\partial_{x}\partial_{\mu}v\left(t,\mathbb{P}_{X_{t}^{\ast}};X_t^{\ast}\right)\right ) \right]\right\}dt=0.
		\end{equation*}
		Recalling Theorem \ref{classical} that $v$ solves (\ref{hjb:1}), we have, for any $t\in[s,T]$,
		\begin{equation*}
			\begin{aligned}
				&-\partial_{t}v(s,\mathbb{P}_{X_t^{\ast}})+ \mathbb{E}\left[G\left ( t,X_t^{\ast},\mathbb{P}_{X_{t}^{\ast}} ,\mathbf{u}^{\ast}_t,-\partial_{\mu}v\left(t,\mathbb{P}_{X_{t}^{\ast}};X_t^{\ast}\right),-\partial_{x}\partial_{\mu}v\left(t,\mathbb{P}_{X_{t}^{\ast}};X_t^{\ast}\right)\right ) \right]\\
				\le&-\partial_{t}v(s,\mathbb{P}_{X_t^{\ast}})+ \sup\limits_{\mathcal{u} \in \mathcal{M}_{t}}\mathbb{E}\left[G\left ( t,X_t^{\ast},\mathbb{P}_{X_{t}^{\ast}} ,\mathcal{u},-\partial_{\mu}v\left(t,\mathbb{P}_{X_{t}^{\ast}};X_t^{\ast}\right),-\partial_{x}\partial_{\mu}v\left(t,\mathbb{P}_{X_{t}^{\ast}};X_t^{\ast}\right)\right ) \right]=0.
				\end{aligned}
		\end{equation*}
		Hence, the above inequality holds with equality for a.e. $t\in[s,T]$, and (\ref{veriii:equ}) follows.
	\end{proof} 
\begin{remark}

In the mean-field control and mean-field games theory, the well-posedness of classical solution to HJB equations and master equations remains a challenging problem. The key results can be found in \cite{Bansil-Meszaros-Mou-Global-well-posedness-of-displacement-monotone-degenerate,Bensoussan-Graber-Yam-Control-on-Hilbert-spaces-and-application,Gangbo-Meszaros-Mou-Zhang-Mean-field-games-master-equations-with-nonseparable-Hamiltonians,Mou-Zhang-Wellposedness-of-second-order-master-equations-with-nonsmooth-data}. It should be noted that the classical solution of HJB equation requires sufficient smoothness of the value function, so it is more practical to introduce the viscosity solution. For the definition and well-posedness of the viscosity solution to HJB equations on the space of probability measures, one can refer to \cite{cosso2024master,cosso2019zero,pham2017dynamic,pham2018bellman,wu2020viscosity}.
\end{remark}

\section{The Connection between MP and DPP}
\label{section4}
The goal of this section is to give the relationship between the MP and DPP for the mean-field stochastic optimal control problem. To our knowledge, there is little literature on this topic, and only \cite{carmona2018probabilistic} studied the case in the heuristic sense where the diffusion term does not depend on the control. In our work, we  generalize the results of \cite{carmona2018probabilistic} in the sense that our system allows the diffusion term to contain control. In addition, we present the relationship between the second-order adjoint equation and the value function. Some examples are also provided to demonstrate our results.\par 
 Let $(X^{\ast},\mathbf{u}^{\ast})$ be an optimal pair, we introduce some simplified notations for the value function $v$ and its derivatives as follows\\
\begin{equation}
\left\{
\begin{aligned}
	&v(t)=v(t,\mathbb{P}_{X_t^{\ast}}),\\
		&v_t(t)=\partial_{t}v(t,\mathbb{P}_{X_t^{\ast}}),\\
	&v_{\mu}(t)=\partial_{\mu}v(t,\mathbb{P}_{X_t^{\ast}};X_t^{\ast}),\\
	&\widetilde{v}_{\mu}(t)=\partial_{\mu}v(t,\mathbb{P}_{X_t^{\ast}};\widetilde{X}_t^{\ast}),\\
	&v_{t\mu}(t)=\partial_{t}\partial_{\mu}v(t,\mathbb{P}_{X_t^{\ast}};X_t^{\ast}),\\
	&v_{x\mu}(t)=\partial_{x}\partial_{\mu}v(t,\mathbb{P}_{X_t^{\ast}};X_t^{\ast}),\\
	&\widetilde{v}_{x\mu}(t)=\partial_{x}\partial_{\mu}v(t,\mathbb{P}_{X_t^{\ast}};\widetilde{X}_t^{\ast}),
\end{aligned}
\right.
\quad
\left\{
\begin{aligned}
	&\widetilde{v}_{\mu\mu}(t)=\partial_{\mu}^{2}v(t,\mathbb{P}_{X_t^{\ast}};X_t^{\ast},\widetilde{X}_t^{\ast}),\\
	&\widetilde{v}_{\mu\mu}^{\ast}(t)=	\partial_{\mu}^{2}v(t,\mathbb{P}_{X_t^{\ast}};\widetilde{X}_t^{\ast},X_t^{\ast}),\\
	&\widetilde{v}_{x\mu\mu}(t)=\partial_{x}\partial_{\mu}^{2}v(t,\mathbb{P}_{X_t^{\ast}};X_t^{\ast},\widetilde{X}_t^{\ast}),\\
	&\widetilde{v}^{\ast}_{x\mu\mu}(t)=\partial_{x}\partial_{\mu}^{2}v(t,\mathbb{P}_{X_t^{\ast}};\widetilde{X}_t^{\ast},X_t^{\ast}),\\
		&\overline{v}_{\mu\mu\mu}(t)=	\partial_{\mu}^{3}v(t,\mathbb{P}_{X_t^{\ast}};X_t^{\ast},\widetilde{X}_t^{\ast},\overline{X}_t^{\ast}),\\
		&\overline{v}_{x\mu\mu\mu}(t)=	\partial_{x}\partial_{\mu}^{3}v(t,\mathbb{P}_{X_t^{\ast}};X_t^{\ast},\widetilde{X}_t^{\ast},\overline{X}_t^{\ast}).
\end{aligned}
\right.
\end{equation} We can define $\widetilde{v}_{y\mu\mu}(t)$, $\widetilde{v}^{\ast}_{y\mu\mu}(t)$, $\widetilde{v}_{\mu x \mu}(t)$, $\widetilde{v}^{\ast}_{\mu x \mu}(t)$  in the similar way as $\widetilde{v}_{x\mu\mu}(t)$ and $\widetilde{v}^{\ast}_{x\mu\mu}(t)$. Then we have the following result about the connection between the first-order adjoint equation and the value function. 
\begin{theorem}\label{Th;4.1}
	Let Assumption \ref{ass:4} hold. For fixed $\left(s,\xi\right)\in\left[0,T\right]\times L^2(\mathcal{F} _{s} ;\mathbb{R}^{d})$,
	let $(X^{\ast},\mathbf{u}^{\ast})$ be an optimal solution of the \textbf{Problem} $\bm{(}\bm{M_{s\xi}}\bm{)}$, and $v$ be the value function. If $v$ $\in$ $C_{b,loc}^{1,2}$$\left(\left[0,T\right]\times\mathcal{P}_{2}(\mathbb{\mathbb{R}}^{d})\right)$, then for a.e. $t\in[s,T]$,
	\begin{equation}\label{conection:2}
		\begin{aligned}
			\partial_tv(t,\mathbb{P}_{X_t^{\ast}})
			&=\mathbb{E}\left[G\left ( t,X_{t}^{\ast},\mathbb{P}_{X_{t}^{\ast}} ,\mathbf{u}_{t}^{\ast},-\partial_{\mu} v(t,\mathbb{P}_{X_{t}^{\ast}};X_{t}^{\ast}),-\partial_{x}\partial_{\mu}v(t,\mathbb{P}_{X_{t}^{\ast}};X_{t}^{\ast}) \right )\right]\\
			&=\sup\limits_{\mathcal{u} \in \mathcal{M}_{t}} \mathbb{E}\left[G\left ( t,X_{t}^{\ast},\mathbb{P}_{X_{t}^{\ast}} ,\mathcal{u},-\partial_{\mu} v(t,\mathbb{P}_{X_{t}^{\ast}};X_{t}^{\ast}),-\partial_{x}\partial_{\mu}v(t,\mathbb{P}_{X_{t}^{\ast}};X_{t}^{\ast}) \right )\right].
		\end{aligned}	 	
	\end{equation}
 Fuethermore, if  $v$ $\in$ $C_{b,loc}^{1,3}\left(\left[0,T\right]\times\mathcal{P}_{2}(\mathbb{\mathbb{R}}^d)\right)$, and $\partial_{t}\partial_{\mu}v$, $\partial_{\mu}\partial_{t}v$ are continuous and bounded, then
	\begin{equation}
		\left\{
		\begin{aligned}
			&\partial_{\mu}v(t,\mathbb{P}_{X_t^{\ast}};X_t^{\ast})=-p_t\qquad\quad\quad\quad\quad \quad\qquad\quad   \text{for all } t\in\left[s,T\right], \mathbb{P}\mbox{-}a.s. ,  \\
			&\partial_{x}\partial_{\mu}v(t,\mathbb{P}_{X_t^{\ast}};X_t^{\ast})\sigma(t,X_t^{\ast},\mathbf{u}_{t}^{\ast},\mathbb{P}_{X_t^{\ast}})=-q_t,\quad\quad a.e.\text{ } t\in\left[s,T\right], \mathbb{P}\mbox{-}a.s.,
		\end{aligned}
		\right. \label{th:1}
	\end{equation}
		where $(p,q)$ solves the first-order adjoint equation (\ref{bansui:1}).
\end{theorem}
\begin{proof} 
From Theorem \ref{ver} (iii), we can obtain $\left(\ref{conection:2}\right)$. Now let us show (\ref{th:1}). By Theorem \ref{classical} and equation $\left(\ref{hjb:1}\right)$, we have
\begin{equation}
	-\partial_{t}v\left(t,\mu\right)+ \mathbb{E}\left[G\left ( t,\gamma,\mu ,\mathbf{u}_{t}^{\ast},-\partial_{\mu} v\left(t,\mu;\gamma\right),-\partial_{x}\partial_{\mu}v\left(t,\mu;\gamma\right) \right )\right]\le 0,\label{proof:1}
\end{equation}
for any $t\in[s,T]$, $ \mu\in\mathcal{P}_2(\mathbb{R}^d)$, with any $\gamma\in L^2(\mathcal{G} ;\mathbb{R}^{d})$ such that $\mathbb{P}_{\gamma}=\mu.$\par
Noting the first equality in (\ref{conection:2}), we have 
\begin{equation}\label{proof:1:2}
		-\partial_{t}v\left(t,\mathbb{P}_{X_t^{\ast}}\right) + 	\mathbb{E}\left[G\left ( t,X_{t}^{\ast},\mathbb{P}_{X_{t}^{\ast}} ,\mathbf{u}_{t}^{\ast},-\partial_{\mu} v(t,\mathbb{P}_{X_{t}^{\ast}};X_{t}^{\ast}),-\partial_{x}\partial_{\mu}v(t,\mathbb{P}_{X_{t}^{\ast}};X_{t}^{\ast}) \right )\right]=0.
\end{equation}
Consequently, if $v\in C_{b,loc}^{1,3}\left(\left[0,T\right]\times\mathcal{P}_{2}(\mathbb{R}^{d})\right)$, for all $t\in[s,T]$, and for any $\varepsilon>0,$ $\theta\in L^{\infty} \left ( \mathcal{G};\mathbb{R}^{d} \right )$, we denote $X_t^{\varepsilon,\theta}:=X_{t}^{\ast}+\varepsilon\theta$, then we have
\begin{equation}\label{first-m-c}
	\begin{aligned}
		&\dfrac{d}{d\varepsilon}\left\{-\partial_{t}v\big(t,\mathbb{P}_{X_t^{\varepsilon,\theta}}\big)\right. \\
		&\left.+ \mathbb{E}\left[ G\big ( t,X_t^{\varepsilon,\theta},\mathbb{P}_{X_t^{\varepsilon,\theta}} ,\mathbf{u}_{t}^{\ast},-\partial_{\mu} v\big(t,\mathbb{P}_{X_t^{\varepsilon,\theta}};X_t^{\varepsilon,\theta}\big),-\partial_{x}\partial_{\mu}v\big(t,\mathbb{P}_{X_t^{\varepsilon,\theta}};X_t^{\varepsilon,\theta}\big) \big)\right]\right\} \bigg| _{\varepsilon=0}=0.
	\end{aligned}	
\end{equation}
Recalling (\ref{Gdefine}), it follows that
\begin{equation*}
	\begin{aligned}
		&\mathbb{E}\left[v_{t\mu}\left(t\right)\cdot\theta  \right]+ \mathbb{E}\bigg[\bigg\{\widetilde{\mathbb{E}}\left[\widetilde{f}_{\mu}^{\ast}\left(t\right)\right]+f_{x}\left(t\right)
		+b_{x}\left(t\right)^{\top}v_{\mu}\left(t\right)
		+\widetilde{\mathbb{E}}\left[\widetilde{b}_{\mu}^{\ast}\left(t\right)^{\top}\widetilde{v}_{\mu}\left(t\right)\right] +v_{x\mu}\left(t\right)^{\top}b\left(t\right)\\
		&
		+\widetilde{\mathbb{E}}\left[\widetilde{v}_{\mu\mu}^{\ast}\left(t\right)^{\top}\widetilde{b}\left(t\right)\right]
		+\frac{1}{2}\sum_{j=1}^{m}\left([\sigma^{j}]_x\left(t\right)\right)^{\top}\left(v_{x\mu}\left(t\right)\sigma\left(t\right)\right)^{j} +\frac{1}{2}\sum_{j=1}^{m}\left([\sigma^{j}]_x\left(t\right)\right)^{\top}\big(v_{x\mu}\left(t\right)^{\top}\sigma\left(t\right)\big)^{j}\\
		&
		+\frac{1}{2}\widetilde{\mathbb{E}}\Big[\sum_{j=1}^{m}\big([\widetilde{\sigma^{j}}]^{\ast}_{\mu}(t)\big)\big(\widetilde{v}_{x\mu}\left(t\right)\widetilde{\sigma}\left(t\right)\big)^{j} \Big]+\frac{1}{2}\widetilde{\mathbb{E}}\Big[\sum_{j=1}^{m}\big([\widetilde{\sigma^{j}}]^{\ast}_{\mu}(t)\big)\big(\widetilde{v}_{x\mu}\left(t\right)^{\top}\widetilde{\sigma}\left(t\right)\big)^{j} \Big]\\
		&
		+\frac{1}{2}tr\big(\sigma\left(t\right)^{\top}v_{xx\mu}\left(t\right)\sigma\left(t\right)\big)\left.\left.+\frac{1}{2}\widetilde{\mathbb{E}}\left[tr\big(\widetilde{\sigma}\left(t\right)^{\top}\widetilde{v}^{\ast}_{\mu x\mu}\left(t\right)\widetilde{\sigma}\left(t\right)\big)\right] \right\}\cdot\theta \right]=0,
	\end{aligned}
\end{equation*}
where
$$
tr\big(\widetilde{\sigma}^{\top}(t)\widetilde{v}^{\ast}_{\mu x\mu}(t)\widetilde{\sigma}(t)\big):=\big(tr(\widetilde{\sigma}^{\top}(t)\widetilde{v}^{\ast}_{\mu^1x\mu}(t)\widetilde{\sigma}(t)),\dots,tr(\sigma^{\top}(t)\widetilde{v}^{\ast}_{\mu^dx\mu}(t)\sigma(t))\big)^{\top},$$
\begin{equation}\label{trv-xxmu}
	tr\big(\sigma^{\top}(t)v_{xx\mu}(t)\sigma(t)\big):=\big(tr(\sigma^{\top}(t)v_{x^1x\mu}(t)\sigma(t)),\dots,tr(\sigma^{\top}(t)v_{x^dx\mu}(t)\sigma(t))\big)^{\top}.
\end{equation}
 From the arbitrariness of $\theta$, we obtain
\begin{equation}\label{first zhankai}
	\begin{aligned}
		&v_{t\mu}\left(t\right)+ \widetilde{\mathbb{E}}\left[\widetilde{f}_{\mu}^{\ast}\left(t\right)\right]+f_{x}\left(t\right)
		+b_{x}\left(t\right)^{\top}v_{\mu}\left(t\right)
		+\widetilde{\mathbb{E}}\left[\widetilde{b}_{\mu}^{\ast}\left(t\right)^{\top}\widetilde{v}_{\mu}\left(t\right)\right] \\
		&+v_{x\mu}\left(t\right)^{\textsuperscript{T}}b\left(t\right)
		+\widetilde{\mathbb{E}}\left[\widetilde{v}_{\mu\mu}^{\ast}\left(t\right)^{\top}\widetilde{b}\left(t\right)\right]
		+\sum_{j=1}^{m}\left([\sigma^{j}]_x\left(t\right)\right)^{\top}\left(v_{x\mu}\left(t\right)\sigma\left(t\right)\right)^{j} \\
		&
		+\widetilde{\mathbb{E}}\Big[\sum_{j=1}^{m}\big([\widetilde{\sigma^{j}}]^{\ast}_{\mu}(t)\big)\big(\widetilde{v}_{x\mu}\left(t\right)\widetilde{\sigma}\left(t\right)\big)^{j} \Big]
		+\frac{1}{2}tr\big(\sigma^{\top}\left(t\right)v_{xx\mu}\left(t\right)\sigma\left(t\right)\big)\\
		&+\frac{1}{2}\widetilde{\mathbb{E}}\left[tr\big(\widetilde{\sigma}^{\top}\left(t\right)\widetilde{v}_{\mu x\mu}^{\ast}\left(t\right)\widetilde{\sigma}\left(t\right)\big)\right]=0,\qquad\qquad \mathbb{P}\mbox{-}a.s.
	\end{aligned}
\end{equation}
Applying It\^o's formula to $v_{\mu}(t)$, we get
\begin{equation}\label{Th1:dvmu}
	\begin{aligned}
		dv_{\mu}\left(t\right) =&v_{t\mu}\left(t\right)dt+\widetilde{\mathbb{E}}\left[\widetilde{v}_{\mu\mu}\left(t\right)\widetilde{b}\left(t\right)\right]dt
		+\widetilde{\mathbb{E}}\Big[\frac{1}{2}tr\big(\widetilde{\sigma}^{\top}\left(t\right)\widetilde{v}_{y\mu\mu}\left(t\right)\widetilde{\sigma}\left(t\right)\big)\Big]dt\\
		&+v_{x\mu}\left(t\right)b\left(t\right)dt+v_{x\mu}\left(t\right)\sigma\left(t\right)dW_t
		+\frac{1}{2}tr\big(\sigma^{\top}\left(t\right)v_{xx\mu}\left(t\right)\sigma\left(t\right)\big)dt,
		\end{aligned}
	\end{equation}
	where $$
	tr\big(\widetilde{\sigma}^{\top}(t)\widetilde{v}_{y\mu\mu}(t)\widetilde{\sigma}(t)\big)=\big(tr\big(\widetilde{\sigma}^{\top}(t)\widetilde{v}_{yy\mu^1}(t)\widetilde{\sigma}(t)\big),\dots,tr\big(\widetilde{\sigma}^{\top}(t)\widetilde{v}_{yy\mu^d}(t)\widetilde{\sigma}(t)\big)\big)^{\top}.
	$$
	In fact, the term $dv_{\mu}\left(t\right)$ is obtained by applying Itô's formula to each component of 
	$v_{\mu}\left(t\right)$. Therefore, one can check that the definition of $tr\big(\sigma^{\top}(t)v_{xx\mu}(t)\sigma(t)\big)$ in (\ref{Th1:dvmu}) should be  $\big(tr(\sigma^{\top}(t)v_{xx\mu^1}(t)\sigma)(t),\dots,tr(\sigma^{\top}(t)\\v_{xx\mu^d}(t)\sigma(t))\big)^{\top},$ which is different from (\ref{trv-xxmu}). However, it can be shown that they are equivalent. Indeed, recalling Lemma \ref{lem:3}, we know $v_{x\mu}(t)$ is a symmetric matrix, thus
	$\partial_{x^i}\partial_{x}\partial_{\mu}v(t,\mathbb{P}_{X_t^{\ast}};X_t^{\ast})=\partial_{x}\partial_{x^i}\partial_{\mu}v(t,\mathbb{P}_{X_t^{\ast}};X_t^{\ast})=\partial_{x}\partial_{x}\partial_{\mu^i}v(t,\mathbb{P}_{X_t^{\ast}};X_t^{\ast}),\text{ for }1\le i \le d,$ which means the equivalence holds.\par
	Moreover, we state that the following equalities hold,
	\begin{equation}\label{Th1:jiaohuan}
		\left\{
		\begin{aligned}
			&v_{t\mu}(t)=v_{\mu t}(t),\\
			&\widetilde{\mathbb{E}}\left[\widetilde{v}_{\mu\mu}^{\ast}\left(t\right)^{\top}\widetilde{b}\left(t\right)\right]=\widetilde{\mathbb{E}}\left[\widetilde{v}_{\mu\mu}\left(t\right)\widetilde{b}\left(t\right)\right],\\
			&\widetilde{\mathbb{E}}\left[tr\big(\widetilde{\sigma}^{\top}\left(t\right)\widetilde{v}_{\mu x\mu}^{\ast}\left(t\right)\widetilde{\sigma}\left(t\right)\big)\right]=\widetilde{\mathbb{E}}\left[tr\big(\widetilde{\sigma}^{\textsuperscript{T}}\left(t\right)\widetilde{v}_{y\mu\mu}\left(t\right)\widetilde{\sigma}\left(t\right)\big)\right].
		\end{aligned}
		\right.
	\end{equation}
	Indeed, recalling Lemma \ref{lem:2}, we have $v_{t\mu}(t)=v_{\mu t}(t)$. Following Lemma \ref{lem:3}, we can obtain 
	$\widetilde{v}_{\mu\mu}\left(t\right)=\partial_{\mu}^{2}v(t,\mathbb{P}_{X_t^{\ast}};X_t^{\ast},\widetilde{X}_t^{\ast})=\partial_{\mu}^{2}v(t,\mathbb{P}_{X_t^{\ast}};\widetilde{X}_t^{\ast},X_t^{\ast})^{\top}=\widetilde{v}_{\mu\mu}^{\ast}\left(t\right)^{\top}
	$
	, which yields $\widetilde{\mathbb{E}}\left[\widetilde{v}_{\mu\mu}^{\ast}\left(t\right)^{\top}\widetilde{b}\left(t\right)\right]=\widetilde{\mathbb{E}}\left[\widetilde{v}_{\mu\mu}\left(t\right)\widetilde{b}\left(t\right)\right]$. Moreover, from Lemma \ref{lem:3}, we also have
	$
	\partial_{\mu}\partial_{\mu^{i}}v(t,\mathbb{P}_{X_t^{\ast}};X_t^{\ast},\widetilde{X}_t^{\ast})=\partial_{\mu^{i}}\partial_{\mu}v(t,\mathbb{P}_{X_t^{\ast}}; \widetilde{X}_t^{\ast},X_t^{\ast})$, for $1\le i \le d.
	$
	It follows directly from the definition of the derivative that
	$
	\partial_{y}\partial_{\mu}\partial_{\mu^{i}}v(t,\mathbb{P}_{X_t^{\ast}};X_t^{\ast},\widetilde{X}_t^{\ast})=\partial_{x}\partial_{\mu^{i}}\partial_{\mu}v(t,\mathbb{P}_{X_t^{\ast}};\\\widetilde{X}_t^{\ast},X_t^{\ast}),1\le i \le d.
	$
	Thus by Lemma \ref{lem:1}, we have
	$
	\widetilde{\mathbb{E}}\left[tr\big(\widetilde{\sigma}^{\top}\left(t\right)\widetilde{v}_{\mu x\mu}^{\ast}\left(t\right)\right.$ $\left.\widetilde{\sigma}(t)\big)\right]=\widetilde{\mathbb{E}}\left[tr\big(\widetilde{\sigma}^{\top}(t)\widetilde{v}_{y\mu\mu}(t)\widetilde{\sigma}(t)\big)\right],
	$
	which means (\ref{Th1:jiaohuan}) holds.\par
	Now, combining (\ref{first zhankai}), (\ref{Th1:dvmu}) and (\ref{Th1:jiaohuan}), we obtain
	\begin{equation*}
		\begin{aligned}
		dv_{\mu}(t)=&-\widetilde{\mathbb{E}}\left[\widetilde{f}_{\mu}^{\ast}\left(t\right)\right]dt-f_{x}\left(t\right)dt-b_{x}\left(t\right)^{\top}v_{\mu}\left(t\right)dt
		\\
		&-\widetilde{\mathbb{E}}\Big[\widetilde{b}_{\mu} ^{\ast}\left(t\right)^{\textsuperscript{T}}\widetilde{v}_{\mu}\left(t\right)\Big] dt-\sum_{j=1}^{m}\left([\sigma^{j}]_x\left(t\right)\right)^{\top}\left(v_{x\mu}\left(t\right)\sigma\left(t\right)\right)^{j}dt\\
		&-\widetilde{\mathbb{E}}\Big[\sum_{j=1}^{m}\left([\widetilde{\sigma}^{j}]_{\mu}^{\ast}\left(t\right)\right)\left(\widetilde{v}_{x\mu}\left(t\right)\widetilde{\sigma}^{\ast}\left(t\right)\right)^{j} \Big]dt
		+v_{x\mu}\left(t\right)\sigma\left(t\right)dB_t.
	\end{aligned}
\end{equation*}
Noting that
$v_{\mu}\left(T\right)=h_x\left(T\right)+\widetilde{\mathbb{E}}\left[\widetilde{h}^{\ast}_{\mu}\left(T\right)\right],
$
by the uniqueness of solution to mean-field type BSDE (\ref{bansui:1}) (see Theorem 3.1 in \cite{buckdahn2009mean}), we obtain
\begin{equation*}
	\left\{
	\begin{aligned}
		&\partial_{\mu}v(t,\mathbb{P}_{X_t^{\ast}};X_t^{\ast})=-p_t\qquad\quad\quad\quad\quad \quad\quad\qquad  \text{for all } t\in\left[s,T\right], \mathbb{P}\mbox{-}a.s. ,  \\
		&\partial_{x}\partial_{\mu}v(t,\mathbb{P}_{X_t^{\ast}};X_t^{\ast})\sigma(t,X_t^{\ast},\mathbf{u}_{t}^{\ast},\mathbb{P}_{X_t^{\ast}})=-q_t,\quad\quad a.e.\text{ } t\in\left[s,T\right], \mathbb{P}\mbox{-}a.s.
	\end{aligned}
	\right. 
\end{equation*}
Thus we have proven (\ref{th:1}).
\end{proof}
In the following, we shall establish the connection between the second-order adjoint equation \eqref{bansui:2} of MP and the derivative of the value function in the case of $d=1$.
\begin{theorem}\label{TH4.2}
	Let Assumption \ref{ass:4} holds, if $v$ $\in$ $C_{b,loc}^{1,4}\left(\left[0,T\right]\times\mathcal{P}_{2}(\mathbb{\mathbb{R}})\right)$, and $\partial_{t}\partial_{x}\partial_{\mu}v$,  $\partial_{x}\partial_{\mu}\partial_{t}v$ are continuous and bounded, then 
	\begin{equation}
		-v_{x\mu}\left(t\right)\ge P_t,\quad \forall t\in\left[s,T\right], \mathbb{P}\mbox{-}a.s.,\label{th:2}
	\end{equation}
	where $(P,Q)$ solves the second-order adjoint equation (\ref{bansui:2}).
\end{theorem}
\begin{proof}
	From (\ref{proof:1}) and (\ref{proof:1:2}), we can also obtain that for all $t\in[s,T]$, for any $\varepsilon>0,$ $\theta\in L^{\infty} \left ( \mathcal{G};\mathbb{R}^{d} \right )$,
	\begin{equation}\label{Bianfen:2}
		\begin{aligned}
			\dfrac{d}{d\varepsilon^2}&\left\{-\partial_{t}v\big(t,\mathbb{P}_{X_t^{\varepsilon,\theta}}\big)\right. \\
			&\left.+ \mathbb{E}\left[ G\big ( t,X_t^{\varepsilon,\theta},\mathbb{P}_{X_t^{\varepsilon,\theta}} ,\mathbf{u}_{t}^{\ast},-\partial_{\mu} v\big(t,\mathbb{P}_{X_t^{\varepsilon,\theta}};X_t^{\varepsilon,\theta}\big),-\partial_{x}\partial_{\mu}v\big(t,\mathbb{P}_{X_t^{\varepsilon,\theta}};X_t^{\varepsilon,\theta}\big) \big )\right]\right\} \bigg| _{\varepsilon=0}\le0,
		\end{aligned}	
	\end{equation}
	where $X_t^{\varepsilon,\theta}:=X_{t}^{\ast}+\varepsilon\theta$.\par
	Since
	\begin{equation}\label{bianfen:2.1.1}
		\dfrac{d}{d\varepsilon^2}\big(\partial_{t}v\big(t,\mathbb{P}_{X_t^{\varepsilon,\theta}}\big) \big)\bigg| _{\varepsilon=0}=\mathbb{E}\left[v_{x\mu t}\left(t\right)\cdot\theta\otimes\theta  \right]+\mathbb{E}\big[\widetilde{\mathbb{E}}\big[\widetilde{v}_{\mu\mu t}\left(t\right)\cdot\theta\otimes\widetilde{\theta}\big]\big].\qquad\qquad
	\end{equation}
	Similarly,
	\begin{equation}
		\begin{aligned}\label{bianfen:2.2}
			\dfrac{d}{d\varepsilon^2}&\mathbb{E}\left[f(t,X_t^{\varepsilon,\theta},\mathbb{P}_{X_t^{\varepsilon,\theta}},\mathbf{u}_t^{\ast})\right]\bigg| _{\varepsilon=0}\\
			&
			=\mathbb{E}\left[\widetilde{\mathbb{E}}\left[\overline{\mathbb{E}}\left[f_{xx}\left(t\right)\cdot\theta\otimes\theta+\widetilde{f}_{y\mu}\left(t\right)\cdot\widetilde{\theta}\otimes\widetilde{\theta}+\left(\widetilde{f}_{\mu x}\left(t\right)+\widetilde{f}_{x\mu}(t)\right)\cdot\theta\otimes\widetilde{\theta}+\overline{f}_{\mu\mu}\left(t\right)\cdot\widetilde{\theta}\otimes\overline{\theta}\right]\right]\right],\qquad\quad
		\end{aligned}
	\end{equation}
	\begin{equation}
		\begin{aligned}\label{bianfen:2.3}
			\dfrac{d}{d\varepsilon^2}&\mathbb{E}\left[b(t,X_t^{\varepsilon,\theta},\mathbb{P}_{X_t^{\varepsilon,\theta}},\mathbf{u}_t^{\ast})\cdot \partial_{\mu}v(t,\mathbb{P}_{X_t^{\varepsilon,\theta}};X_t^{\varepsilon,\theta})\right]\bigg| _{\varepsilon=0}\\
			&=\mathbb{E}\left[\widetilde{\mathbb{E}}\left[A_b^1(t)\cdot\theta\otimes\theta+A_b^2(t)\cdot\widetilde{\theta}\otimes\widetilde{\theta}\right]\right]+\mathbb{E}\left[\widetilde{\mathbb{E}}\left[\overline{\mathbb{E}}\left[K_b^1(t)\cdot\theta\otimes\widetilde{\theta}+K_b^2\cdot\widetilde{\theta}\otimes\overline{\theta}\right]\right]\right],\qquad\qquad\quad\qquad\quad
		\end{aligned}
	\end{equation}
	where 
	\begin{equation*}
		\begin{aligned}
			A_b^1(t):=&b_{xx}\left(t\right)v_{\mu}\left(t\right)+2b_{x}\left(t\right)v_{x\mu}\left(t\right)+v_{xx\mu}\left(t\right)b\left(t\right),\\
			A_b^2(t):=&\widetilde{b}_{y\mu}\left(t\right)v_{\mu}\left(t\right)+\widetilde{v}_{y\mu\mu}(t)b\left(t\right),\\
			K_b^1(t):=&\widetilde{b}_{\mu x}\left(t\right)v_{\mu}\left(t\right)+2b_{x}\left(t\right)\widetilde{v}_{\mu\mu}\left(t\right)+\widetilde{b}_{x\mu}\left(t\right)v_{\mu}\left(t\right)+2\widetilde{b}_{\mu}\left(t\right)v_{x\mu}\left(t\right)+\widetilde{v}_{\mu x\mu}(t)b\left(t\right) ,\\
			K_b^2(t):=& \overline{b}_{\mu\mu}\left(t\right)v_{\mu}\left(t\right)+\widetilde{b}_{\mu}\left(t\right)\partial_{\mu}^{2}v\big(t,\mathbb{P}_{X_{t}^{\ast}};X_{t}^{\ast},\overline{X}_{t}^{\ast}\big)+\overline{v}_{\mu\mu\mu}(t)b\left(t\right)+\widetilde{v}_{\mu\mu}\left(t\right)^{\top}\partial_{\mu}b\big(t,X_{t}^{\ast},\mathbb{P}_{X_{t}^{\ast}},u_{t}^{\ast};\overline{X}_{t}^{\ast}\big).
		\end{aligned}
	\end{equation*}
	We also have
	\begin{equation}
		\begin{aligned}\label{bianfen:2.4}
			\dfrac{d}{d\varepsilon^2}&\mathbb{E}\left[\frac{1}{2}\partial_{x}\partial_{\mu}v(t,\mathbb{P}_{X_t^{\varepsilon,\theta}};X_t^{\varepsilon,\theta})\sigma \big ( t,X_t^{\varepsilon,\theta},\mathbb{P}_{X_t^{\varepsilon,\theta}},\mathbf{u}_t^{\ast}\big )\sigma^{\top}\big ( t,X_t^{\varepsilon,\theta},\mathbb{P}_{X_t^{\varepsilon,\theta}},\mathbf{u}_t^{\ast}\big )\right]\bigg| _{\varepsilon=0}\\
			&=\mathbb{E}\left[\widetilde{\mathbb{E}}\left[A_{\sigma}^1(t)\cdot\theta\otimes\theta+A_{\sigma}^2(t)\cdot\widetilde{\theta}\otimes\widetilde{\theta}\right]\right]+\mathbb{E}\left[\widetilde{\mathbb{E}}\left[\overline{\mathbb{E}}\left[K_{\sigma}^1(t)\cdot\theta\otimes\widetilde{\theta}+K_{\sigma}^2\cdot\widetilde{\theta}\otimes\overline{\theta}\right]\right]\right],
		\end{aligned}
	\end{equation}
	where
	\begin{equation*}
		\begin{aligned}
			A_{\sigma}^1(t):=&\frac{1}{2}\sigma\left(t\right)\sigma^{\top}\left(t\right)v_{xxx\mu}\left(t\right)+\sigma_{xx}\left(t\right)v_{x\mu}\left(t\right)\sigma^{\top}\left(t\right)		+\sigma_x\left(t\right)\sigma^{\top}_x\left(t\right)v_{x\mu}\left(t\right)+2\sigma_x\left(t\right)v_{xx\mu}\left(t\right)\sigma^{\top}\left(t\right),\\
			A_{\sigma}^2(t):=&\frac{1}{2}\sigma\left(t\right)\sigma\left(t\right)^{\top}\widetilde{v}_{y\mu x\mu}\left(t\right)+\widetilde{\sigma}_{y\mu}\left(t\right)v_{x\mu}\left(t\right)\sigma^{\top}\left(t\right),\\
			K_{\sigma}^1(t):=&\frac{1}{2}\sigma\left(t\right)\sigma\left(t\right)^{\top}\widetilde{v}_{\mu xx\mu}\left(t\right)+\frac{1}{2}\sigma\left(t\right)\sigma\left(t\right)^{\top}\widetilde{v}_{x\mu x\mu}\left(t\right)+2\sigma_{x}\left(t\right)v_{x\mu}\left(t\right)\widetilde{\sigma}_{\mu}^{\top}\left(t\right)+\widetilde{\sigma}_{\mu x}\left(t\right)v_{x\mu}\left(t\right)\sigma^{\top}\left(t\right)\\
			&+2\sigma_{x}\left(t\right)\widetilde{v}_{\mu x\mu}(t)\sigma^{\top}\left(t\right)+\widetilde{\sigma}_{x\mu}\left(t\right)v_{x\mu}\left(t\right)\sigma^{\top}\left(t\right)+2\widetilde{\sigma}_{\mu}\left(t\right)v_{xx\mu}\left(t\right)\sigma^{\top}\left(t\right),\\
			K_{\sigma}^2(t):=&\frac{1}{2}\sigma\left(t\right)\sigma^{\top}\left(t\right)\overline{v}_{\mu\mu x\mu}\left(t\right)+\overline{\sigma}_{\mu\mu}\left(t\right)v_{x\mu}\left(t\right)\sigma^{\top}\left(t\right)+\widetilde{\sigma}_{\mu}\left(t\right)\partial_{\mu}\partial_{x}\partial_{\mu}v(t,\mathbb{P}_{X_{t}^{\ast}};X_{t}^{\ast},\overline{X}_{t}^{\ast})\sigma^{\top}\left(t\right)\\
			&+\widetilde{\sigma}_{\mu}\left(t\right)v_{x\mu}\left(t\right)\partial_{\mu}\sigma(t,X_{t}^{\ast},\mathbb{P}_{X_{t}^{\ast}},\mathbf{u}_{t}^{\ast};\overline{X}_{t}^{\ast})^{\top}+\partial_{\mu}\partial_{x}\partial_{\mu}v\big(t,\mathbb{P}_{X_{t}^{\ast}};X_{t}^{\ast},\overline{X_{t}^{\ast}}\big)\sigma\left(t\right)\widetilde{\sigma}_{\mu}\left(t\right)^{\top}.
		\end{aligned}
	\end{equation*}
	 Therefore, from (\ref{Bianfen:2}) and combining (\ref{bianfen:2.1.1})-(\ref{bianfen:2.4}), we can obtain
	\begin{equation}
		\begin{aligned}\label{bianfen:2.1}
			&-\mathbb{E}\left[\widetilde{\mathbb{E}}\left[\left(v_{x\mu t}(t)+f_{xx}(t)+A_b^1(t)+A_{\sigma}^1(t)\right)\cdot\theta\otimes\theta+\big(\widetilde{f}_{y\mu}\left(t\right)+A_b^{2}(t)+A_{\sigma}^2(t)\big)\cdot\widetilde{\theta}\otimes\widetilde{\theta}\right]\right]\\
			&-\mathbb{E}\left[\widetilde{\mathbb{E}}\left[\big(\widetilde{v}_{\mu\mu t}\left(t\right)+\widetilde{f}_{\mu x}\left(t\right)+\widetilde{f}_{x\mu}(t)+K_b^1(t)+K_{\sigma}^1\big)\cdot\theta\otimes\widetilde{\theta}\right]\right]\\
			&-\mathbb{E}\left[\widetilde{\mathbb{E}}\left[\overline{\mathbb{E}}\left[\left(\overline{f}_{\mu\mu}\left(t\right)+K_b^2(t)+K_{\sigma}^2\right)\cdot\widetilde{\theta}\otimes\overline{\theta}\right]\right]\right]\le0.
		\end{aligned}
	\end{equation}
	Due to the arbitrariness of $\theta$, we can take it as $\alpha \Gamma$, where $\Gamma$ is arbitrary random variable belonging to $L^{\infty} \left ( \mathcal{G};\mathbb{R}^{d} \right )$, and $\alpha$ is independent of $(X_{t}^{\ast}, u_t^{\ast},\Gamma)$ with $\mathbb{P} (\alpha=1)=\mathbb{P}(\alpha=-1)=\frac{1}{2}$. We point out that the existence of $\alpha$ is guaranteed by Remark \ref{rem:independ}. Then, we can get rid of the terms containing $\theta\otimes\widetilde{\theta}$ and $\widetilde{\theta}\otimes\overline{\theta}$ in (\ref{bianfen:2.1}). Indeed, when we take $\theta$ as $\alpha \Gamma$, (\ref{bianfen:2.1}) can be written as
	\begin{equation}
		\begin{aligned}
			&-\mathbb{E}\left[\widetilde{\mathbb{E}}\left[\left(v_{x\mu t}(t)+f_{xx}(t)+A_b^1(t)+A_{\sigma}^1(t)\right)\cdot \Gamma\otimes\Gamma\right]\right]\mathbb{E}\left[\alpha^2\right]\\
			&-\mathbb{E}\left[\widetilde{\mathbb{E}}\left[\big(\widetilde{f}_{y\mu}\left(t\right)+A_b^{2}(t)
			+A_{\sigma}^2(t)\big)\cdot\widetilde{\Gamma}\otimes\widetilde{\Gamma}\right]\right]\widetilde{\mathbb{E}}\Big[\widetilde{\alpha}^2\Big]\\
			&-\mathbb{E}\left[\widetilde{\mathbb{E}}\left[\big(\widetilde{v}_{\mu\mu t}\left(t\right)+\widetilde{f}_{\mu x}\left(t\right)+\widetilde{f}_{x\mu}(t)+K_b^1(t)+K_{\sigma}^1\big)\cdot\Gamma\otimes\widetilde{\Gamma}\right]\right]\mathbb{E}\Big[\alpha\Big]\widetilde{\mathbb{E}}\Big[\widetilde{\alpha}\Big]\\
			&-\mathbb{E}\left[\widetilde{\mathbb{E}}\left[\overline{\mathbb{E}}\left[\left(\overline{f}_{\mu\mu}\left(t\right)+K_b^2(t)+K_{\sigma}^2\right)\cdot\widetilde{\Gamma}\otimes\overline{\Gamma}\right]\right]\right]\widetilde{\mathbb{E}}\Big[\widetilde{\alpha}\Big]\overline{\mathbb{E}}\Big[\overline{\alpha}\Big]\le0.
		\end{aligned}
	\end{equation}
	Thus, from that $\mathbb{E}\left[\alpha^2\right]=1$ and $\mathbb{E}\left[\alpha\right]=0$, we obtain 
	\begin{equation*}
		\begin{aligned}
			&-\mathbb{E}\left[\widetilde{\mathbb{E}}\left[\left(v_{x\mu t}(t)+f_{xx}(t)+A_b^1(t)+A_{\sigma}^1(t)\right)\cdot \Gamma\otimes\Gamma\right]\right]\\
			&-\mathbb{E}\left[\widetilde{\mathbb{E}}\left[\big(\widetilde{f}_{y\mu}\left(t\right)+A_b^{2}(t)
			+A_{\sigma}^2(t)\big)\cdot\widetilde{\Gamma}\otimes\widetilde{\Gamma}\right]\right]\le0.
		\end{aligned}
	\end{equation*}
	By applying Fubini's theorem, we have
	\begin{equation}\label{eq-fubini}
		\begin{aligned}
			&-\mathbb{E}\left[\widetilde{\mathbb{E}}\left[\left(v_{x\mu t}(t)+f_{xx}(t)+A_b^1(t)+A_{\sigma}^1(t)+\widetilde{f}^{\ast}_{y\mu}\left(t\right)+\widetilde{A}_b^{2}(t)
			+\widetilde{A}_{\sigma}^2(t)\right)\cdot \Gamma\otimes\Gamma\right]\right]\le0,
		\end{aligned}
	\end{equation}
	where
		\begin{equation*}
		\begin{aligned}
			\widetilde{A}_b^2(t):=&\widetilde{b}^{\ast}_{y\mu}\left(t\right)\widetilde{v}_{\mu}\left(t\right)+\widetilde{v}^{\ast}_{y\mu\mu}(t)\widetilde{b}\left(t\right),\\
			\widetilde{A}_{\sigma}^2(t):=&\frac{1}{2}\widetilde{\sigma}\left(t\right)\widetilde{\sigma}\left(t\right)^{\top}\widetilde{v}^{\ast}_{y\mu x\mu}\left(t\right)+\widetilde{\sigma}^{\ast}_{y\mu}\left(t\right)\widetilde{v}_{x\mu}\left(t\right)\widetilde{\sigma}^{\top}\left(t\right).
		\end{aligned}
	\end{equation*}
	 Recalling the definition (\ref{dao:3}) and the connection (\ref{th:1}), and following (\ref{eq-fubini}) and the arbitrariness of $\Gamma$, we get $\mathbb{P}$-almost surely that
	\begin{equation}
		\begin{aligned}
			&
			-2b_{x}\left(t\right)v_{x\mu}\left(t\right)-\sigma_x\left(t\right)\sigma_x^{\top}\left(t\right)v_{x\mu}\left(t\right) -2v_{xx\mu}\left(t\right)\sigma\left(t\right)\sigma^{\top}_x\left(t\right)+H_{x x}(t)+\mathbb{E}\left[\widetilde{H}_{ y\mu}^*(t)\right]\\
			&
			\le v_{x\mu t}\left(t\right)+v_{xx\mu}\left(t\right)b\left(t\right)+\frac{1}{2}\sigma\left(t\right)\sigma^{\top}\left(t\right)v_{xxx\mu}\left(t\right)+\widetilde{\mathbb{E}}\big[\widetilde{v}^{\ast}_{y\mu\mu}(t)\widetilde{b}\left(t\right)\big]+\frac{1}{2}\widetilde{\mathbb{E}}\big[\widetilde{\sigma}\left(t\right)\widetilde{\sigma}^{\top}\left(t\right)\widetilde{v}^{\ast}_{y\mu x\mu}\left(t\right)\big].
		\end{aligned}\label{proofinequ:1}
	\end{equation}
	
	Now, by applying It\^o's formula to $-v_{x\mu}(t)$ we get
	\begin{equation}\label{campare:2}
		\begin{aligned} 
	-d\left(-v_{x\mu}(t)\right)=&\bigg\{v_{tx\mu}\left(t\right)+v_{xx\mu}\left(t\right)b\left(t\right)
			+\dfrac{1}{2}\sigma\left(t\right)\sigma^{\top}\left(t\right)v_{xxx\mu}\left(t\right)+\widetilde{\mathbb{E}}\left[ \widetilde{v}_{\mu x\mu}(t)\widetilde{b}\left(t\right)\right]
			\\
			&+\dfrac{1}{2}\widetilde{\mathbb{E}}\left[\widetilde{\sigma}\left(t\right)\widetilde{\sigma}^{\top}\left(t\right)\widetilde{v}_{y\mu x\mu}\left(t\right)\right]\bigg\}dt-\sum_{j=1}^{m}\big(-v_{xx\mu}\left(t\right)\sigma^{j}\left(t\right)\big)dB_t^j.
		\end{aligned}
	\end{equation}
	Similarly to (\ref{Th1:jiaohuan}), we can obtain
	\begin{equation}\label{th:zhuanhuan2}
		\left\{
		\begin{aligned}
			&v_{x\mu t}\left(t\right)=v_{tx\mu}\left(t\right),\\
			&\widetilde{\mathbb{E}}\big[\widetilde{v}^{\ast}_{y\mu\mu}(t)\widetilde{b}\left(t\right)\big]=\widetilde{\mathbb{E}}\left[ \widetilde{v}_{\mu x\mu}(t)\widetilde{b}\left(t\right)\right],\\
			&\widetilde{\mathbb{E}}\big[\widetilde{\sigma}\left(t\right)\widetilde{\sigma}^{\top}\left(t\right)\widetilde{v}^{\ast}_{y\mu x\mu}\left(t\right)\big]=\widetilde{\mathbb{E}}\left[\widetilde{\sigma}\left(t\right)\widetilde{\sigma}^{\top}\left(t\right)\widetilde{v}_{y\mu x\mu}\left(t\right)\right].
		\end{aligned}
		\right.
	\end{equation}
	
	Finally, note that $P_T=-v_{x\mu}\left(T\right)=-h_{x x}(T)-\widetilde{\mathbb{E}}\left[\tilde{h}_{\mu y}^*(T)\right]$. Following (\ref{proofinequ:1}) and (\ref{th:zhuanhuan2}), and using the comparison theorem for scalar BSDEs (see Theorem 2.2 of \cite{Pengl1997backwardinfinance}) to (\ref{bansui:2}) and (\ref{campare:2}), we obtain (\ref{th:2}).
\end{proof}
\begin{remark}\label{matrix-BSDE-difficult}
	The proof of the above theorem relies on the comparison theorem for BSDEs. When $d>1$, we notice that both sides of the inequality (\ref{proofinequ:1}) are symmetric matrices. Theorem 3.1 in \cite{hu2006comparison} established the comparison theorem for matrix-valued BSDEs, but due to the complexity of matrix derivatives, we are currently unable to verify whether the conditions of Theorem 3.1 in \cite{hu2006comparison} are satisfied. Therefore, the above theorem can only be given for the case of $d=1$, and the general case is still a challenging problem.
\end{remark}
\begin{remark}\label{tiaojian}
	The assumptions for coefficients $f$ and $h$ in Theorem \ref{Th;4.1} and Theorem \ref{TH4.2} can be suitably relaxed when the initial value $\xi$ of \textbf{Problem} $\bm{(}\bm{M_{s\xi}}\bm{)}$ belongs to $L^\infty(\mathcal{F}_s;\mathbb{R})$. Specifically, let coefficients $b$, $\sigma$ still satisfy Assumption \ref{ass:4}, and $f$, $h$ satisfy the following assumptions.\\
	\textbf{Assumption $\bm{(}\bm{A_{f,h}}\bm{)}$}.
	\begin{enumerate}[(i)]
		\item $f$, $h$ are jointly measurable in all variables, and locally bounded in $(x,\mu)$ uniformly w.r.t. $(t,u)$. Furthermore, there exists a locally bounded function $g:[0,+\infty)\mapsto [0,+\infty)$ such that
		$	\abs{f(t,x,\mu,u)}\le g(W_2(\mu,\delta_0))(1+\abs{x}^2)$, as well as $h(x,\mu)\le g(W_2(\mu,\delta_0))(1+\abs{x}^2)$, for all $(t,x,\mu,u)\in[0,T]\times\mathbb{R}^d\times\mathcal{P}_2(\mathbb{R}^d)\times U$;
		\item $f(t,\cdot, \cdot, u)$, $h(\cdot, \cdot) \in \mathbb{C}_{b,loc}^{2,2}(\mathbb{R}^d \times \mathcal{P}_2(\mathbb{R}^d) )$, for all $(t,u) \in [0,T]\times U$. Moreover, all the derivatives of $f$ and $h$ are   Lipschitz
		continuous with Lipschitz constant independent of $(t,u)$. Additionally, these derivatives are locally bounded in $(x,\mu)$ uniformly w.r.t. $(t,u,y,z)$.
	\end{enumerate}
	Under the above assumptions, we claim that Theorem \ref{MP} and Theorem \ref{DDP:1} still hold. In fact, on the one hand, the maximum principle obtained from \cite{Nie-Yan-Mean-field-partial-information-non-zero-sum} shows that the conditions of $\partial_{x}f$, $\partial_{\mu}f$, $\partial_{x}h$ and $\partial_{\mu}h$ can be extended to be linear growth w.r.t. $(x,\mu)$, which is guaranteed here. In addition, the proof in \cite{Nie-Yan-Mean-field-partial-information-non-zero-sum} relies heavily on the estimation that $\mathbb{E}\left[\sup_{t \in [s,T]}\abs{X_t}^{2k}\right]\le C\varepsilon^k$, $1\le k\le 4$, which can be obtained directly under the assumption that $\xi\in L^\infty(\mathcal{F}_s;\mathbb{R})$. On the other hand, one can check that the assumptions of Corollary 3.9 in \cite{cosso2023optimal} remain valid. Therefore, it can be shown that the proofs of Theorem \ref{Th;4.1} and Theorem \ref{TH4.2} continue to hold when the conditions of $f$ and $h$ in Assumption \ref{ass:4} are relaxed to Assumption $(A_{f,h})$.
\end{remark}
\begin{remark}
	We claim that if the inequality in (\ref{Bianfen:2}) is strictly valid, then the strict inequality \(-\partial_{x}\partial_{\mu}v(t,\mathbb{P}_{X_{t}^{\ast}},X_{t}^{\ast})>P_t\) will occur. In the following examples where we suppose that both the state and Brownian motion are 1-dimensional, i.e., $d=m=1$, we can obtain this result through the strict comparison theorem for scalar BSDEs.
\end{remark}
\begin{example}\label{exa:1}
	Consider the following control system: 
	\begin{equation}
		\left\{\begin{aligned}
			d X_t & =2 \mathbf{u}_t d t+\sqrt{2} d B_t,\qquad t \in[s, T], \\
			X_s & =\xi,\qquad\qquad \xi\in L^\infty(\mathcal{F}_s,\mathbb{R}) ,
		\end{aligned}\right.\label{exa1:1}
	\end{equation}
	with the control domain being $U=[-1,1]$ and the cost functional being 
	\begin{equation}
		\begin{aligned}
			J\left(s,\xi,\mathbf{u}\right)=\mathbb{E}\Big[\int_{s}^{T} 
			&2\mathbb{E}\left[lnch(X_t)+1\right](\mathbf{u}^2_t+1)dr-(lnchX_T+1)\mathbb{E}\left[lnchX_T+1\right]\Big], \label{exa1:2}
		\end{aligned}
	\end{equation}
	where $chx=\dfrac{1}{2}(e^x+e^{-x})$. Thus, we have $h(x,\mu)=-(lnchx+1) 
	\int_{\mathbb{R}}(lnchx+1)\mu(dx)$ and  $f(t,x,\mu,u)=2(u^2+1)\int_{\mathbb{R}}(lnchx+1)\mu(dx)$. One can check that the coefficients $f$, $h$ satisfy the Assumption $(A_{f,h})$ in Remark \ref{tiaojian}.
	For any fixed $(s,\xi)\in[0,T]\times L^\infty(\mathcal{F}_s,\mathbb{R})$ and $\mathbf{u}\in\mathcal{U}$, applying It\^o's formula to the process $(lnch(X_t)+1)\mathbb{E}\left[lnch(X_t)+1\right]$ and combining with (\ref{exa1:1}), we get
	\begin{equation*}
		\begin{aligned}
			&d(lnch(X_t)+1)\mathbb{E}\left[lnch(X_t)+1\right]\\
			=&\left\{th(X_t)\left(2\mathbf{u}_tdt+\sqrt{2}dB_t\right)+\left[ch(X_t)\right]^{-2}dt\right\}\mathbb{E}\left[lnch(X_t)+1\right]\\
			&+(lnch(X_t)+1)\mathbb{E}\left[2th(X_t)\mathbf{u}_t+\left[ch(X_t)\right]^{-2}\right]dt,
		\end{aligned}
	\end{equation*}
	where $th(x)=\dfrac{e^x-e^{-x}}{e^x+e^{-x}}$.
	Then
	\begin{equation*}
		\begin{aligned}
			&\mathbb{E}^2\left[lnch(X_T)+1\right]-\mathbb{E}^2\left[lnch(\xi)+1\right]\\
			=&\mathbb{E}\left[\int_{s}^{T}\left\{2th(X_t)\mathbf{u}_t+\left[ch(X_t)\right]^{-2}\right\}\mathbb{E}\left[lnch(X_t)+1\right]dt\right]
		\end{aligned}
	\end{equation*}	
	\begin{equation*}
		\begin{aligned}
			&+\mathbb{E}\left[\int_{s}^{T}(lnch(X_t)+1)\mathbb{E}\left[2th(X_t)\mathbf{u}_t+\left[ch(X_t)\right]^{-2}\right]dt\right]\\
			=&\mathbb{E}\left[\int_{s}^{T}\left\{4th(X_t)\mathbf{u}_t+2\left[ch(X_t)\right]^{-2}\right\}\mathbb{E}\left[lnch(X_t)+1\right]dt\right].
		\end{aligned}
	\end{equation*}
	From (\ref{exa1:2}) and $th^{2}(x)+ch^{-2}(x)=1 ,$ for any $x\in\mathbb{R}$, we get
	\begin{equation}\label{ex1-cost-function-trans}
		\begin{aligned}
			&J\left(s,\xi,\mathbf{u}\right)+\mathbb{E}^2\left[lnch(\xi)+1\right]\\
			&=\mathbb{E}\left[\int_{s}^{T}\mathbb{E}\left[lnch(X_t)+1\right]\left\{2\mathbf{u}^2_t+2-4th(X_t)\mathbf{u}_t-2\left[ch(X_t)\right]^{-2}\right\}\right]dt\\
			&=2\mathbb{E}\left[\int_{s}^{T}\mathbb{E}\left[lnch(X_t)+1\right]\left\{\mathbf{u}_t-th(X_t)\right\}^2\right]dt\ge0,
		\end{aligned}
	\end{equation}
	where $X$ is the solution of (\ref{exa1:1}).  The inequality \eqref{ex1-cost-function-trans} will become equality if and only if $\mathbf{u}_t=th(X_t)$. Thus it follows that 
	$
		v(s,\mathbb{P}_{\xi})=-\mathbb{E}^2\left[lnch(\xi)+1\right],
	$
	which belongs to $C_{b,loc}^{1,4}$$\left(\left[0,T\right]\times\mathcal{P}_{2}(\mathbb{R})\right)$. One can check that $\partial_{t}v(t,\mu)\equiv0$, $\partial_{\mu}v(t,\mu;x)=-2th(x)\int_{\mathbb{R}}(lnch(x)+1)\mu(dx)$ and $\partial_{x}\partial_{\mu}v(s,\mu;x)=-2\left[ch(x)\right]^{-2}\int_{\mathbb{R}}(lnch(x)+1)\mu(dx)$.
	
	Now let $(X^{\ast},\mathbf{u}^{\ast})$ be an optimal pair for the related mean-field control problem, where $X^{\ast}$ solves (\ref{exa1:1}) and $\mathbf{u}^{\ast}_t= th(X^{\ast}_t)$. 
	To give the first-order adjoint equation, one can check that
	\begin{equation*}
		\left\{
		\begin{aligned}
			&b_x(t)=\widetilde{b}^{\ast}_{\mu}(t)=\sigma_x(t)=\widetilde{\sigma}^{\ast}_{\mu}(t)=f_x(t)=0,\\
			&\widetilde{\mathbb{E}}\left[\widetilde{f}_{\mu}^{\ast}\left(t\right)\right]=2\widetilde{\mathbb{E}}\left[\left\{1+[th(\widetilde{X}^{\ast}_t)]^2\right\}th(X^{\ast}_t)\right] =2\mathbb{E}\left[1+[th(X^{\ast}_t)]^2\right]th(X^{\ast}_t),\\
			&h_x(t)=-th(X^{\ast}_t)\mathbb{E}\left[lnch(X^{\ast}_t)+1\right],\\
			&\widetilde{\mathbb{E}}\left[\widetilde{h}_{\mu}^{\ast}\left(t\right)\right]=-\widetilde{\mathbb{E}}\left[(lnch(\widetilde{X}^{\ast}_t)+1)th(X^{\ast}_t)\right]=-th(X^{\ast}_t)\mathbb{E}\left[lnch(X^{\ast}_t)+1\right].
		\end{aligned}
		\right.
	\end{equation*}
	Then, it follows that
	\begin{equation}
		\left\{
		\begin{aligned}
			dp_{t} =
			&2\mathbb{E}\left[1+[th(X^{\ast}_t)]^2\right]th(X^{\ast}_t) dt+q_{t}dB_{t},
			\\
			p_{T}=&2th(X^{\ast}_T)\mathbb{E}\left[lnch(X^{\ast}_T)+1\right].
		\end{aligned}
		\right.\label{exa1 bansui:1}
	\end{equation}
	Inspired by the terminal form of (\ref{exa1 bansui:1}), we apply It\^o's formula to $th(X^{\ast}_t) \mathbb{E}\left[lnch(X^{\ast}_t)+1\right]$, it follows that
	\begin{equation*}
		\begin{aligned}
			&d th(X^{\ast}_t)\mathbb{E}\left[lnch(X^{\ast}_t)+1\right]\\
			=&\mathbb{E}\left[lnch(X^{\ast}_t)+1\right]\left[ch(X^{\ast}_t)\right]^{-2}\left[\left(2\mathbf{u}^{\ast}_t-2th(X_t^{\ast})\right)dt+\sqrt{2}dB_t\right]\\
			&+th(X^{\ast}_t)\mathbb{E}\left[2th(X^{\ast}_t)\mathbf{u}^{\ast}_t+\left[ch(X^{\ast}_t)\right]^{-2}\right]dt.
		\end{aligned}
	\end{equation*}
	The uniqueness of the adapted solution $(p,q)$ to the first-order adjoint equation  (\ref{exa1 bansui:1}) yields 
	\begin{equation*}
		\left\{
		\begin{aligned}
			&p_t=2\mathbb{E}\left[lnch(X_t^{\ast})+1\right]th(X^{\ast}_t),\\
			&q_t=2\sqrt{2}\left[ch(X^{\ast}_t)\right]^{-2}\mathbb{E}\left[lnch(X^{\ast}_t)+1\right].
		\end{aligned}\right.
	\end{equation*}
	Therefore (\ref{th:1}) holds. To give the second-order adjoint equation, one can check that 
	\begin{equation*}
		\left\{
		\begin{aligned}
			&H_{xx}(t)=0,\\
			&\widetilde{\mathbb{E}}\left[\widetilde{H}_{y\mu}^{\ast}(t)\right]=-\widetilde{\mathbb{E}}\left[\widetilde{f}_{y\mu}^{\ast}(t)\right]=-2\widetilde{\mathbb{E}}\left[\left\{1+[th(\widetilde{X}^{\ast}_t)]^2\right\}[ch(X^{\ast}_t)]^{-2}\right],\\
			&\widetilde{\mathbb{E}}\left[\widetilde{h}_{y\mu}^{\ast}(t)\right]=-\widetilde{\mathbb{E}}\left[(lnch(\widetilde{X}^{\ast}_t)+1)[ch(X^{\ast}_t)]^{-2}\right],\\	
			&h_{xx}(t)=-\mathbb{E}\left[(lnch(\widetilde{X}^{\ast}_t)+1)\right][ch(X^{\ast}_t)]^{-2}.
		\end{aligned}
		\right.
	\end{equation*}
	 Then, we get 
	\begin{equation}
		\left\{\begin{aligned}
			d P_t= 
			& 2\mathbb{E}\left[\left\{1+[th(X^{\ast}_t)]^2\right\}\right][ch(X^{\ast}_t)]^{-2}d t+Q_t d B_t ,\\
			P_T= &2\mathbb{E}\left[lnch(X^{\ast}_T)+1\right] [ch(X^{\ast}_T)]^{-2}.
		\end{aligned}\right.\label{bansui:3}
	\end{equation}
	Next, we denote\\$$\mathcal{P}_t:=-v_{x\mu}(t,\mathbb{P}_{X^{\ast}_t};X^{\ast}_t)=2[ch(X^{\ast}_t)]^{-2}\mathbb{E}\left[lnch(X^{\ast}_t)+1\right] ,$$    $$\mathcal{Q}_t:=-4\sqrt{2}[ch(X^{\ast}_t)]^{-3}sh(X^{\ast}_t)\mathbb{E}\left[lnch(X^{\ast}_t)+1\right].$$
	\\Applying It\^o's formula to $\mathcal{P}_t$ and using $sh(x)=th(x)ch(x)= \frac{1}{2}(e^x-e^{-x})$, we obtain
	\begin{equation*}
		\begin{aligned}
			d\mathcal{P}_t=&2\mathbb{E}\left[lnch(X^{\ast}_t)+1\right]\left\{-2[ch(X_t^{\ast})]^{-3}sh(X_t^{\ast})\left[2th(X_t^{\ast})dt+\sqrt{2}dB_t\right]\right.\\
			&+\left[6[ch(X_t^{\ast})]^{-4}[sh(X_t^{\ast})]^{2}-2[ch(X_t^{\ast})]^{-2}\right]dt\Big\}+2[ch(X^{\ast}_t)]^{-2}\mathbb{E}\left[2[th(X^{\ast}_t)]^{2}+[ch(X^{\ast}_t)]^{-2}\right]dt\\
			=&2\mathbb{E}\left[lnch(X^{\ast}_t)+1\right]\left\{\left[-4[ch(X_t^{\ast})]^{-2}[th(X_t^{\ast})]^{2}+6[ch(X_t^{\ast})]^{-2}[th(X_t^{\ast})]^{2}\right.\right.\\
			&\left.\left.-2[ch(X_t^{\ast})]^{-2}\right]dt-2\sqrt{2}[ch(X_t^{\ast})]^{-3}sh(X_t^{\ast})dB_t\right\}+2[ch(X^{\ast}_t)]^{-2}\mathbb{E}\left[[th(X^{\ast}_t)]^{2}+1\right]dt\\
			=&\left\{-4\mathbb{E}\left[lnch(X^{\ast}_t)+1\right][ch(X^{\ast}_t)]^{-4}+2[ch(X^{\ast}_t)]^{-2}\mathbb{E}\left[[th(X^{\ast}_t)]^{2}+1\right]\right\}dt\\
			&-4\sqrt{2}\mathbb{E}\left[lnch(X^{\ast}_t)+1\right][ch(X_t^{\ast})]^{-3}sh(X_t^{\ast})dB_t.
		\end{aligned}
	\end{equation*}
	Thus we have
	\begin{equation}
		\left\{\begin{aligned}
			d\mathcal{P}_t= 
			&g(t,X^{\ast}_t,\mathbb{P}_{X^{\ast}_t})d t+ \mathcal{Q}_t d B_t ,\\
			\mathcal{P}_T= &2\mathbb{E}\left[lnch(X^{\ast}_T)+1\right] [ch(X^{\ast}_T)]^{-2},
		\end{aligned}\right.\label{bansui:4}
	\end{equation}
	where $g(t,x,\mu)=-4[ch(x)]^{-4}\int_{\mathbb{R}}(lnch(x)+1)\mu(x)+2ch(x)^{-2}\int_{\mathbb{R}}(1+[th(x)]^2)\mu(x)$, and one can check that $g(t,X^{\ast}_t,\mathbb{P}_{X^{\ast}_t})- 2\mathbb{E}\left[1+[th(X^{\ast}_t)]^2\right][ch(X^{\ast}_t)]^{-2}=-4\mathbb{E}\left[lnch(X^{\ast}_t)+1\right]ch(X^{\ast}_t)^{-4}<0$. By comparing (\ref{bansui:3}) and (\ref{bansui:4}) which are standard BSDEs, the strictly comparison theorem for scalar BSDEs yields that $P_t< -v_{x\mu}\left(t\right)$.\par
	From another perspective, by studying (\ref{Bianfen:2}), for any $\gamma$, $\bar{\gamma}\in L^2(\mathcal{F}_t;\mathbb{R})$, we define
	\begin{equation}
		\begin{aligned}
			\Phi(t, \mathbb{P}_{\gamma},\mathbf{u})&:= -\partial_{t}v(t,\mathbb{P}_{\gamma})+\mathbb{E}\left[G\left(t,\gamma,\mathbb{P}_{\gamma},\mathbf{u}_t,-\partial_{\mu}v\left(t,\mu;\gamma\right),-\partial_{x}\partial_{\mu}v\left(t,\mu;\gamma\right)\right)\right]\\
			&=-2\mathbb{E}\left[lnch(\gamma)+1\right]\mathbb{E}\left[(\mathbf{u}_t-th(\gamma))^2\right].
		\end{aligned}
	\end{equation}
	Then, one can check that 
	 \begin{equation*}
		\left\{
		\begin{aligned}
			\dfrac{d}{d\epsilon}\Phi(t, \mathbb{P}_{\gamma+\epsilon\bar{\gamma}},\mathbf{u})\big| _{\epsilon=0}=&-2\mathbb{E}\left[th(\gamma)\cdot\bar{\gamma}\right]\mathbb{E}\left[(\mathbf{u}_t-th(\gamma))^2\right]\\
			&-4\mathbb{E}\left[lnch(\gamma)+1\right]\mathbb{E}\left[(th(\gamma)-\mathbf{u}_t)[ch(\gamma)]^{-2}\cdot\bar{\gamma}\right],\\
			\dfrac{d}{d\epsilon^2}\Phi(t, \mathbb{P}_{\gamma+\epsilon\bar{\gamma}},\mathbf{u})\big| _{\epsilon=0}=&-2\mathbb{E}\left[[ch(\gamma)]^{-2}\cdot\bar{\gamma}^2\right]\mathbb{E}\left[\left(\mathbf{u}_t-th(\gamma)^2\right)\right]\\
			&-8\mathbb{E}\left[th(\gamma)\bar{\gamma}\right]\mathbb{E}\left[(th(\gamma)-\mathbf{u}_t)[ch(\gamma)]^{-2}\cdot\bar{\gamma}\right]\\
			&-4\mathbb{E}\left[lnch(\gamma)+1\right]\mathbb{E}\left[\left\{[ch(\gamma)]^{-4}-2[ch(\gamma)]^{-2}th(\gamma)\left(th(\gamma)-\mathbf{u}_t\right)\right\}\cdot\bar{\gamma}^2\right].
		\end{aligned}
		\right.
	\end{equation*} 
	Obviously, for any $\gamma$, $\bar{\gamma}\in L^2(\mathcal{F}_t;\mathbb{R})$, if $\mathbf{u}_t=th(\gamma)$, we have $	\dfrac{d}{d\epsilon}\Phi(t, \mathbb{P}_{\gamma+\epsilon\bar{\gamma}},\mathbf{u})\big| _{\epsilon=0}=0$ and $	\dfrac{d}{d\epsilon^2}\Phi(t, \mathbb{P}_{\gamma+\epsilon\bar{\gamma}},\\ \mathbf{u})\big| _{\epsilon=0}=-4\mathbb{E}\left[lnch(\gamma)+1\right]\mathbb{E}\left[[ch(\gamma)]^{-4}\cdot\bar{\gamma}^2\right]<0$. Recalling that $(X^{\ast},\mathbf{u}^{\ast})$ is an optimal pair for related mean-field control problem, where $X^{\ast}$ satisfies (\ref{exa1:1}) and $\mathbf{u}^{\ast}_t= th(X^{\ast}_t)$, thus we get 
	\begin{equation*}
		\left\{
		\begin{aligned}
			&\dfrac{d}{d\epsilon}\Phi(t, \mathbb{P}_{X_t^{\ast}+\epsilon\bar{\gamma}},\mathbf{u}_t^{\ast})\big| _{\epsilon=0}=0,\\
			&\dfrac{d}{d\epsilon^2}\Phi(t, \mathbb{P}_{X_t^{\ast}+\epsilon\bar{\gamma}},\mathbf{u}_t^{\ast})\big| _{\epsilon=0}<0.
		\end{aligned}
		\right.
	\end{equation*} 
	Thus, following the proofs of Theorem \ref{Th;4.1} and Theorem \ref{TH4.2}, we also have (\ref{th:1}) and show $P_t<-v_{x\mu}(t)$.
\end{example}
\begin{remark}
	We can also observe that if the second-order derivative of (\ref{Bianfen:2}) at $\mu=X_t^{\ast}$ equals zero or inequality (\ref{proofinequ:1}) becomes equality, then we can get $P_t= -v_{x\mu}\left(t\right)$ for all $ t\in\left[s,T\right], \mathbb{P}\mbox{-}a.s.$ and $Q_t=-v_{x\mu\mu}(t)\sigma(t)$ for a.e. $t\in\left[s,T\right],$ $\mathbb{P}\mbox{-}a.s.$ The following example is just in such  situation. 
\end{remark}
\begin{example}\label{exa:2}
	Consider the following control system:
	\begin{equation}
		\left\{\begin{aligned}
			d X_t & = \mathbf{u}_t d t+ \mathbf{u}_t d B_t,\qquad t \in[s, T], \\
			X_s & =\xi,\qquad\qquad \xi\in L^\infty(\mathcal{F}_s,\mathbb{R}) ,
		\end{aligned}\right.\label{exa2:1}
	\end{equation}
	with the control domain being $U=[-3,-2]\cup[1,2]$ and the cost functional being 
	\begin{equation}
		\begin{aligned}
			J\left(s,\xi,\mathbf{u}\right)=&\mathbb{E}\left[\int_{s}^{T} \mathbb{E}\left[lnch(X_t)+1\right](-\mathbf{u}^2_t[th(X_t)]^2+2\mathbf{u}_tth(X_t)+2\mathbf{u}_t^2+2\mathbf{u}_t)dr\right. \\
			&-(lnch(X_T)+1)\mathbb{E}\left[lnch(X_T)+1\right]\bigg] .\label{exa2:2}
		\end{aligned}
	\end{equation}
	Thus $h(x,\mu)=-(lnch(x)+1) \int_{\mathbb{R}}(lnch(x)+1)\mu(x)$ and $f(t,x,\mu,u)=(-u^2[th(x)]^2+2uth(x)+2u^2+2u)\int_{\mathbb{R}}(lnch(x)+1)\mu(x)$. Similarly to Example \ref{exa:1}, the coefficients $f$, $h$ satisfy the Assumption $(A_{f,h})$ in Remark \ref{tiaojian}.
	For any
	$(s,\xi) \in [0,T]\times L^\infty(\mathcal{F}_s,\mathbb{R})$ and $\mathbf{u}\in\mathcal{U},$ applying It\^o's formula to the process $(lnch(X_t)+1)\mathbb{E}\left[(lnch(X_t)+1)\right]$ and combining with (\ref{exa2:1}), we get
	\begin{equation*}
		\begin{aligned}
			&d(lnch(X_t)+1)\mathbb{E}\left[lnch(X_t)+1\right]\\
			=&\left\{th(X_t)\left(\mathbf{u}_tdt+u_tdB_t\right)+\frac{1}{2}\left[ch(X_t)\right]^{-2}\mathbf{u}_t^{2}dt\right\}\mathbb{E}\left[lnch(X_t)+1\right]\\
			&+(lnch(X_t)+1)\mathbb{E}\left[th(X_t)\mathbf{u}_t+\frac{1}{2}\left[ch(X_t)\right]^{-2}\mathbf{u}_t^2\right]dt.
		\end{aligned}
	\end{equation*}
	Then
	\begin{equation*}
		\begin{aligned}
			&\mathbb{E}^2\left[lnch(X_T)+1\right]-	\mathbb{E}^2\left[lnch(\xi)+1\right]\\
			=&\mathbb{E}\left[\int_{s}^{T}\mathbb{E}\left[lnch(X_t)+1\right]\left\{th(X_t)\mathbf{u}_t+\frac{1}{2}\left[ch(X_t)\right]^{-2}\mathbf{u}_t^{2}\right\}dt\right]\\
			&+\mathbb{E}\left[\int_{s}^{T}\left(lnch(X_t)+1\right)\mathbb{E}\left[th(X_t)\mathbf{u}_t+\frac{1}{2}\left[ch(X_t)\right]^{-2}\mathbf{u}_t^2\right]dt\right]\\
			=&\mathbb{E}\left[\int_{s}^{T}\mathbb{E}\left[lnch(X_t)+1\right]\left\{2th(X_t)\mathbf{u}_t+\left[ch(X_t)\right]^{-2}\mathbf{u}_t^{2}\right\}dt\right].
		\end{aligned}
	\end{equation*}
	Thus, from (\ref{exa2:2}), we obtain $J\left(s,\xi,\mathbf{u}\right)+\mathbb{E}^2\left[lnch(\xi)+1\right]=\mathbb{E}\left[\int_{s}^{T}\mathbb{E}\left[lnch(X_t)+1\right]\left(\mathbf{u}_t^2+2\mathbf{u}_t\right)dt\right].$
Noting that $U=[-3,-2]\cup[1,2]$, we have 
$\mathbb{E}\left[\int_{s}^{T}\mathbb{E}\left[lnch(X_t)+1\right]\left(\mathbf{u}_t^2+2\mathbf{u}_t\right)dt\right]\ge0,$
and equality holds if and only if $\mathbf{u}_t=-2$. It follows that 
	$
		v(s,\mathbb{P}_{\xi})=-\mathbb{E}^2\left[lnch(\xi)+1\right],
	$
	which belongs to $C_{b,loc}^{1,4}\left(\left[0,T\right]\times\mathcal{P}_{2}(\mathbb{R})\right)$. Thus, one can check that $\partial_{t}v(t,\mu)\equiv0$, $\partial_{\mu}v(t,\mu;x)=-2th(x)\int_{\mathbb{R}}(lnch(x)+1)\mu(dx)$ and $\partial_{x}\partial_{\mu}v(s,\mu;x)=-2\left[ch(x)\right]^{-2}\int_{\mathbb{R}}(lnch(x)+1)\mu(dx)$. \par Now let $(X^{\ast},\mathbf{u}^{\ast})$ be an optimal pair to related mean-field control problem, where $X^{\ast}$ solves (\ref{exa2:1}) with $\mathbf{u}_t^{\ast}= -2$. 
	To give the first-order adjoint equaiton, one can check that
\begin{equation*}
	\left\{
	\begin{aligned}	
	&	b_x(t)=\widetilde{b}_{\mu}^{\ast}(t)=\sigma_x(t)=\widetilde{\sigma}_{\mu}^{\ast}(t)=0,\\
	&f_x(t)=-4[ch(X_t^{\ast})]^{-2}\left\{1+2th(X_t^{\ast})\right\}\mathbb{E}\left[lnch(X_t^{\ast})+1\right],\\	
		&\widetilde{\mathbb{E}}\left[\widetilde{f}_{\mu}^{\ast}\left(t\right)\right]=4\widetilde{\mathbb{E}}\left[\left\{[ch(\widetilde{X}_t^{\ast})]^{-2}-th(\widetilde{X}_t^{\ast})\right\}th(X^{\ast}_t)\right]=4\mathbb{E}\left[[ch(X_t^{\ast})]^{-2}-th(X_t^{\ast})\right]th(X^{\ast}_t),\\
			&h_x(t)=-th(X^{\ast}_t)\mathbb{E}\left[lnch(X^{\ast}_t)+1\right],\\
			&	\widetilde{\mathbb{E}}\left[\widetilde{h}_{\mu}^{\ast}\left(t\right)\right]=-\widetilde{\mathbb{E}}\left[(lnch(\widetilde{X}^{\ast}_t)+1)th(X^{\ast}_T)\right]=-th(X^{\ast}_t)\mathbb{E}\left[lnch(X^{\ast}_t)+1\right].
		\end{aligned}
		\right.	
	\end{equation*}
	Then, it follows that
	\begin{equation}
		\left\{
		\begin{aligned}
			dp_{t} =
			&4\mathbb{E}\left[[ch(X_t^{\ast})]^{-2}-th(X_t^{\ast})\right]th(X^{\ast}_t)\\
			&-4[ch(X_t^{\ast})]^{-2}\left\{1+2th(X_t^{\ast})\right\}\mathbb{E}\left[lnch(X_t^{\ast})+1\right] dt
			+q_{t}dB_{t},
			\\
			p_{T}=&2th(X^{\ast}_T)\mathbb{E}\left[lnch(X^{\ast}_T)+1\right].
		\end{aligned}
		\right.\label{exa2 bansui:1}
	\end{equation}
	To solve (\ref{exa2 bansui:1}), motivated by the terminal form of (\ref{exa2 bansui:1}), we apply It\^o's formula to $2th(X_t^{\ast})\mathbb{E}[lnch(X_t^{\ast})+1]$, and it follows that
	\begin{equation*}
		\begin{aligned}
			&d\left(2th(X_t^{\ast})\mathbb{E}\left[lnch(X_t^{\ast})+1\right]\right)\\
			=&-4[ch(X_t^{\ast})]^{-2}\left\{1+2th(X_t^{\ast})\right\}\mathbb{E}\left[lnch(X_t^{\ast})+1\right] \\
			&+4th(X_t^{\ast})\mathbb{E}\left[[ch(X_t^{\ast})]^{-2}-th(X_t^{\ast})\right]dt-4\mathbb{E}\left[lnch(X_t^{\ast})+1\right][ch(X_t^{\ast})]^{-2}dB_t.\\
		\end{aligned}
	\end{equation*}
	From the uniqueness of the solution to equation (\ref{exa2 bansui:1}), we obtain
	\begin{equation*}
		\left\{
		\begin{aligned}
			&p_t=2th(X_t^{\ast})\mathbb{E}\left[lnch(X_t^{\ast})+1\right],\\
			&q_t=-4[ch(X_t^{\ast})]^{-2}\mathbb{E}\left[lnch(X_t^{\ast})+1\right].
		\end{aligned}\right.
	\end{equation*}
	Therefore (\ref{th:1}) hold. To give the second-order adjoint, one can check that \\
	\begin{equation*}
		\left\{
		\begin{aligned}
			&H_{xx}(t)=-f_{xx}(t)=-\left\{8[ch(X_t^{\ast})]^{-2}th(X_t^{\ast})\left(1+2th(X_t^{\ast})\right)-8[ch(X_t^{\ast})]^{-4}\right\}\mathbb{E}\left[lnch(X_t^{\ast})+1\right]\\
			&\qquad\hspace{0.5em}\hspace{0.3em}=-\mathbb{E}\left[lnch(X_t^{\ast})+1\right]\left\{8[ch(X_t^{\ast})]^{-2}th(X_t^{\ast})+16[ch(X_t^{\ast})]^{-2}-24[ch(X_t^{\ast})]^{-4}\right\},\\
			&\widetilde{\mathbb{E}}\left[\widetilde{H}_{y\mu}^{\ast}(t)\right]=-\widetilde{\mathbb{E}}\left[\widetilde{f}_{y\mu}^{\ast}(t)\right]=-4\widetilde{\mathbb{E}}\left[\left\{[ch(\widetilde{X}_T^{\ast})]^{-2}-th(\widetilde{X}_t^{\ast})\right\}[ch(X^{\ast}_t)]^{-2}\right],\\
			&h_{xx}(t)=-\mathbb{E}\left[lnch(X^{\ast}_t)+1\right] [ch(X^{\ast}_t)]^{-2},\\
			&\widetilde{\mathbb{E}}\left[\widetilde{h}_{y\mu}^{\ast}(t)\right]=-\widetilde{\mathbb{E}}\left[(lnch(\widetilde{X}^{\ast}_t)+1)[ch(X^{\ast}_t)]^{-2}\right].
		\end{aligned}
		\right.
	\end{equation*}
	Then, we get 
	\begin{equation}
		\left\{\begin{aligned}
			d P_t= 
			&g_1(t,X^{\ast}_t,\mathbb{P}_{X^{\ast}_t})d t+Q_t d B_t, \\
			P_T= &2\mathbb{E}\left[lnch(X^{\ast}_T)+1\right] [ch(X^{\ast}_T)]^{-2},
		\end{aligned}\right.\label{exp2:bansui:3}
	\end{equation}
	where
	\begin{flalign*}
		\quad\qquad\quad	g_1(t,x,\mu)=&4[ch(x)]^{-2}\int_{\mathbb{R}}\left([ch(x)]^{-2}-th(x)\right)\mu(dx)+\\
			&8\left(-3[ch(x)]^{-4}+2[ch(x)]^{-2}+th(x)[ch(x)]^{-2}\right)\int_{\mathbb{R}}(lnch(x)+1)\mu(dx).
	\end{flalign*}
	By applying It\^o’s formula to $-v_{x\mu}(t)$, we obtain
	\begin{equation*}
		\begin{aligned}
			&d\left(2[ch(X_t^{\ast})]^{-2}\mathbb{E}\left[lnch(X_t^{\ast})+1\right]\right)\\
			=&\left\{-4[ch(X_t^{\ast})]^{-2}th(X_t^{\ast})(-2dt-2dB_t)-8\left(-2[ch(X_t^{\ast})]^{-2}[th(X_t^{\ast})]^{-2}\right.\right.\\
			&\left.\left.+[ch(X_t^{\ast})]^{-4}\right)dt\right\}\mathbb{E}\left[lnch(X_t^{\ast})+1\right]+4[ch(X_t^{\ast})]^{-2}\mathbb{E}\left[[ch(X_t^{\ast})]^{-2}-th(X_t^{\ast})\right]dt\\
			=&\mathbb{E}\left[lnch(X_t^{\ast})+1\right]\left\{8[ch(X_t^{\ast})]^{-2}th(X_t^{\ast})+16[ch(X_t^{\ast})]^{-2}-24[ch(X_t^{\ast})]^{-4}\right\}\\
			&+4[ch(X_t^{\ast})]^{-2}\mathbb{E}\left[[ch(X_t^{\ast})]^{-2}-th(X_t^{\ast})\right]dt+8\mathbb{E}\left[lnch(X_t^{\ast})+1\right][ch(X_t^{\ast})]^{-2}th(X_t^{\ast})dB_t.
		\end{aligned}
	\end{equation*}
	Thus, from the uniqueness of the solution to (\ref{exp2:bansui:3}), we have
	\begin{equation}
		\left\{
		\begin{aligned}
			&P_t=2[ch(X_t^{\ast})]^{-2}\mathbb{E}\left[lnch(X_t^{\ast})+1\right],\\
			&Q_t=8[ch(X_t^{\ast})]^{-2}th(X_t^{\ast})\mathbb{E}\left[lnch(X_t^{\ast})+1\right],
		\end{aligned}\right.
	\end{equation}
	which yields
	\begin{equation*}
		\left\{
		\begin{aligned}
			P_t&=-v_{x\mu}(t), \qquad\qquad\quad \text{for all }  t\in\left[s,T\right], \mathbb{P}\mbox{-}a.s.\\
			Q_t&=-v_{xx\mu}(t)\sigma(t), \qquad\qquad a.e. \text{ } t\in\left[s,T\right], \mathbb{P}\mbox{-}a.s.
		\end{aligned}
		\right.
	\end{equation*}
	From another perspective, for any $\gamma$, $\bar{\gamma}\in L^2(\mathcal{F}_t;\mathbb{R})$, we define
	 \begin{equation}
	 	\begin{aligned}
	 		\Phi(t, \mathbb{P}_{\gamma},\mathbf{u}):&= -\partial_{t}v(t,\mathbb{P}_{\gamma})+\mathbb{E}\left[G\left(t,\gamma,\mathbb{P}_{\gamma},\mathbf{u}_t,-\partial_{\mu}v\left(t,\mathbb{P}_{\gamma},\gamma\right),-\partial_{x}\partial_{\mu}v\left(t,\mathbb{P}_{\gamma},\gamma\right)\right)\right]\\
	 		&=\mathbb{E}\left[lnch(\gamma)+1\right]\mathbb{E}\left[{-\mathbf{u}_t}^{2} -2\mathbf{u}_t\right].
	 	\end{aligned}
	 \end{equation}
	 Obviously,
	 \begin{equation*}
	 	\left\{
	 	\begin{aligned}
	 		&\dfrac{d}{d\epsilon}\Phi(t, \mathbb{P}_{\gamma+\epsilon\bar{\gamma}},\mathbf{u})\big| _{\epsilon=0}=\mathbb{E}\left[th(\gamma)\cdot\bar{\gamma}\right]\mathbb{E}\left[{-\mathbf{u}_t}^{2} -2\mathbf{u}_t\right],\\
	 		&\dfrac{d}{d\epsilon^2}\Phi(t, \mathbb{P}_{\gamma+\epsilon\bar{\gamma}},\mathbf{u})\big| _{\epsilon=0}=\mathbb{E}\left[[ch(\gamma)]^{-2}\cdot\bar{\gamma}^2\right]\mathbb{E}\left[{-\mathbf{u}_t}^{2} -2\mathbf{u}_t\right].
	 	\end{aligned}
	 	\right.
	 \end{equation*} 
	 Noting that $U=[-3,-2]\cup[1,2]$, we know if $\mathbf{u}_t=\mathbf{u}_t^{\ast}=-2$, then $\dfrac{d}{d\epsilon}\Phi(t, \mathbb{P}_{\gamma+\epsilon\bar{\gamma}},\mathbf{u})\big| _{\epsilon=0}=\dfrac{d}{d\epsilon^2}\Phi(t, \mathbb{P}_{\gamma+\epsilon\bar{\gamma}},\\ \mathbf{u})\big| _{\epsilon=0}\equiv0$.
	 Thus, (\ref{Bianfen:2}) should
	hold with the equality, which can make $(\ref{proofinequ:1})$ hold with equality, then we can get 
	 \begin{equation*}
	 	\left\{
	 	\begin{aligned}
	 		P_t&=-v_{x\mu}(t), \qquad\qquad\quad  \text{for all }  t\in\left[s,T\right], \mathbb{P}\mbox{-}a.s.\\
	 		Q_t&=-v_{x\mu\mu}(t)\sigma(t), \qquad\qquad a.e. \text{ } t\in\left[s,T\right], \mathbb{P}\mbox{-}a.s.
	 	\end{aligned}
	 	\right.
	 \end{equation*}
\section{Conclusion}
In this paper, we develop the connection between the MP and DPP for general mean-field optimal control problems. We emphasize that the diffusion term of the state equation in this paper can depend on the control. We establish the relationship between the value function and the first- and second-order adjoint equations via the BSDE comparison theorem, and obtain the general verification theorem. \par
As shown in Theorem \ref{Th;4.1} and Theorem \ref{TH4.2}, we need to assume that the value function has sufficient smoothness, which is not always satisfied. Hence, in the future, on the one hand,  we hope to study the well-posedness of classical solution to HJB equation (\ref{hjb:1}). On the other hand, we shall investigate the nonsmooth versions of the connection  between the MP and DPP.
\end{example}
\bibliographystyle{plain}
\bibliography{references}
\end{document}